\pdfoutput=1
\documentclass[11pt]{amsart}
\usepackage{amsmath}
\usepackage{amssymb}
\usepackage{harpoon}
\usepackage[pdftex]{graphicx}
\usepackage{esint}
\usepackage[latin1]{inputenc}
\usepackage{xcolor}

\definecolor{brass}{rgb}{0.71, 0.65, 0.36}

\usepackage{tikz-cd}
\theoremstyle{plain}
\newtheorem{theorem}{Theorem}[section]
\newtheorem{lemma}[theorem]{Lemma}
\newtheorem{prop}[theorem]{Proposition}
\newtheorem{cor}[theorem]{Corollary}

\theoremstyle{definition}
\newtheorem{remark}[theorem]{Remark}

\numberwithin{equation}{section}
\def\Ric{\operatorname{Ric}}

\def\inf{\operatorname{inf}}

\theoremstyle{plain}

\numberwithin{equation}{section}

\allowdisplaybreaks

\begin{document}

\title[Spacetime Harmonic Functions and Mass]{Spacetime Harmonic Functions and the Mass of 3-Dimensional Asymptotically Flat Initial Data for the Einstein Equations}

\author[Hirsch]{Sven Hirsch}
\address{
Department of Mathematics\\
Duke University\\
Durham, NC, 27708\\
USA}
\email{sven.hirsch@duke.edu, demetre.kazaras@duke.edu}

\author[Kazaras]{Demetre Kazaras}
\author[Khuri]{Marcus Khuri}
\address{
Department of Mathematics\\
Stony Brook University \\
Stony Brook, NY, 11794-3660\\
USA}
\email{demetre.kazaras@stonybrook.edu, khuri@math.sunysb.edu}

\thanks{D. Kazaras acknowledges the support of NSF Grant DMS-1547145. M. Khuri acknowledges the support of NSF Grant DMS-1708798, and Simons Foundation Fellowship 681443.}

\begin{abstract}
We give a lower bound for the Lorentz length of the ADM energy-momentum vector (ADM mass)
of 3-dimensional asymptotically flat initial data sets for the Einstein equations. The bound is given in terms of linear growth `spacetime harmonic functions' in addition to the energy-momentum density of matter fields, and is valid regardless of whether the dominant energy condition holds or whether the data possess a boundary. A corollary of this result is a new proof of the spacetime positive mass theorem for complete initial data or those with weakly trapped surface boundary, and includes the rigidity statement which asserts that the mass vanishes if and only if the data arise from Minkowski space. The proof has some analogy with both the Witten spinorial approach as well as the marginally outer trapped surface (MOTS) method of Eichmair, Huang, Lee, and Schoen. Furthermore, this paper generalizes the harmonic level set technique used in the Riemannian case by Bray, Stern, and the second and third authors, albeit with a different class of level sets. Thus, even in the time-symmetric (Riemannian) case a new inequality is achieved.
\end{abstract}

\maketitle

\section{Introduction}
\label{sec1} \setcounter{equation}{0}
\setcounter{section}{1}

Let $(M,g,k)$ be a smooth connected 3-dimensional initial data set for the Einstein equations. This represents an embedded spacelike hypersurface in spacetime, so that $g$ is a Riemannian metric and $k$ is a symmetric 2-tensor denoting the extrinsic curvature.
These objects satisfy the constraint equations
\begin{equation}
\mu=\frac{1}{2}\left(R_g +(\mathrm{Tr}_g k)^2 -|k|_g^2\right),\quad\quad
J=\mathrm{div}_g\left(k-(\mathrm{Tr}_g k)g\right),
\end{equation}
where $R_g$ is the scalar curvature and $\mu$ and $J$ represent the energy and momentum density of the matter fields. It will be assumed that the data are asymptotically flat.
This means that there is a compact set $\mathcal{C}\subset M$ such that $M\setminus \mathcal{C}=\cup_{\ell=1}^{\ell_0}M_{end}^{\ell}$ where the ends $M_{end}^\ell$ are pairwise disjoint and diffeomorphic to the complement of a ball $\mathbb{R}^3 \setminus B_1$, and there exists in each end a coordinate system satisfying
\begin{equation}\label{asymflat}
|\partial^l (g_{ij}-\delta_{ij})(x)|=O(|x|^{-q-l}),\quad l=0,1,2,\quad\quad
|\partial^l k_{ij}(x)|=O(|x|^{-q-1-l}),\quad l=0,1,
\end{equation}
for some $q>\tfrac{1}{2}$. The energy and momentum densities will be taken to be integrable $\mu, J \in L^1(M)$ so that the ADM energy and linear momentum of each end is well-defined \cite{Bartnik,Chrusciel,OMurchadha} and given by
\begin{equation}
E=\lim_{r\rightarrow\infty}\frac{1}{16\pi}\int_{S_{r}}\sum_i \left(g_{ij,i}-g_{ii,j}\right)\upsilon^j dA,\quad\quad
P_i=\lim_{r\rightarrow\infty}\frac{1}{8\pi}\int_{S_{r}} \left(k_{ij}-(\mathrm{Tr}_g k)g_{ij}\right)\upsilon^j dA,
\end{equation}
where $\upsilon$ is the unit outer normal to the coordinate sphere $S_r$ of radius $r=|x|$ and $dA$ denotes its area element. The ADM mass $m=\sqrt{E^2-|P|^2}$ is the Lorentz length of the ADM energy-momentum vector $(E,P)$. If the dominant energy condition is satisfied $\mu\geq |J|_g$, then the spacetime positive mass theorem asserts that the ADM energy-momentum is nonspacelike, and characterizes Minkowski space as the unique spacetime having asymptotically flat initial data with vanishing mass.

\begin{theorem}\label{positivemass}
Let $(M,g,k)$ be a complete and asymptotically flat initial data set for the Einstein equations satisfying the dominant energy condition. Then in each end $E\geq |P|$, and $E=|P|$ in some end if and only if $E=|P|=0$ and the data arise from an isometric embedding into Minkowski space.
\end{theorem}

Slightly less general incarnations of this theorem were first established in the early 1980's by Schoen and Yau \cite{SchoenYauI,SchoenYauII,SchoenYauIII}, and independently by Witten \cite{ParkerTaubes,Witten}. The later approach of Witten utilized the hypersurface Dirac operator and generalized Lichnerowicz formula to establish the positive mass inequality $E\geq |P|$. In \cite{Witten} an outline was given for the rigidity statement in vacuum, and under various stronger hypotheses this was established by Ashtekar and Horowitz \cite{AshtekarHorowitz}, and Yip \cite{Yip}. A complete and rigorous proof was then given by Beig and Chru\'{s}ciel \cite{BeigChrusciel}
with the asymptotic assumption that
$\mu, |J|_g =O(|x|^{-q-5/2})$,
instead of integrability.
In \cite{SchoenYauI} Schoen and Yau proved the time-symmetric case, when $k=0$, via a contradiction argument employing stable minimal hypersurfaces. The non-time-symmetric case was then reduced to the previous result by solving Jang's equation \cite{SchoenYauII}. This established nonnegativity of the energy $E\geq 0$ along with the case of equality, under the additional asymptotic hypothesis \cite{Eichmair} that
$\mathrm{Tr}_g k=O(|x|^{-q-3/2})$.
Nonnegativity of the energy in fact implies the inequality $E\geq |P|$, through the boost result of Christodoulou and \'{O} Murchadha \cite{ChristodoulouOMurchadha} and subsequent generalization of the boost argument in \cite{EichmairHuangLeeSchoen}. Moreover, assuming
the decay for the matter energy-momentum density present in \cite{BeigChrusciel},
Eichmair, Huang, Lee, and Schoen \cite{EichmairHuangLeeSchoen} have shown that stable marginally outer trapped surfaces (MOTS) may be used to prove $E\geq |P|$, in analogy with the time-symmetric minimal hypersurface proof. Although the rigidity statement was not treated in \cite{EichmairHuangLeeSchoen}, the work of Huang and Lee \cite{HuangLee} demonstrates that it follows from the inequality $E\geq |P|$.

In higher dimensions the approach of Witten generalizes for spin manifolds \cite{Bartnik,Ding}, with the case of equality being settled by Chru\'{s}ciel and Maerten in \cite{ChruscielMaerten}. Similar to the minimal hypersurface technique, the MOTS method \cite{EichmairHuangLeeSchoen} extends without difficulty to dimensions $3\leq d\leq 7$, and Eichmair \cite{Eichmair} has generalized the Jang deformation to these dimensions as well. Combined then with the rigidity argument of Huang and Lee \cite{HuangLee}, which holds in all dimensions, the result holds up to dimension 7 without the spin assumption.
A compactification argument has been given by Lohkamp \cite{Lohkamp2}, akin to the Riemannian case \cite{Lohkamp}, in which the spacetime positive mass inequality reduces to the nonexistence of initial data of the form $(N^d \# T^d,g,k)$ satisfying a strict dominant energy condition, where $N^d$ is a compact manifold and $T^d$ is the torus; this relies also on the boost theorem.
Furthermore, we point out the articles of Schoen and Yau \cite{SchoenYauIV} and Lohkamp \cite{Lohkamp1} which address the higher dimensional Riemannian problem.
For a survey of topics related to the positive mass theorem see the book by Lee \cite{Lee}.

The purpose of the current article is to give a lower bound for the difference $E-|P|$
in terms of linear growth `spacetime harmonic functions' and the difference of energy-momentum densities for the matter fields $\mu-|J|_g$. In order to state the main result, let $\Sigma$ be a closed 2-sided hypersurface in $M$ with null expansions $\theta_{\pm}=H\pm \mathrm{Tr}_{\Sigma}k$, where $H$ is the mean curvature of $\Sigma$ with respect to the unit normal $\upsilon$ pointing towards infinity in a designated end $M_{end}$. The null expansions are the mean curvatures in the null directions $\upsilon\pm n$ when viewed as codimension two surfaces in spacetime, where $n$ is the future pointing timelike normal to the slice $(M,g,k)$. Physically these quantities may be interpreted as determining the rate at which the area of a shell of light is changing
as it moves away from the surface in the outward future/past direction, and thus can be used to measure the strength of the gravitational field. The gravitational field is strong if $\Sigma$ is \textit{outer or inner trapped}, that is $\theta_{+}<0$ or $\theta_{-}<0$. Moreover, $\Sigma$ is called a \textit{marginally outer or inner trapped surface} (MOTS or MITS) if $\theta_+ =0$ or $\theta_- =0$; in the literature these surfaces are also sometimes referred to as future or past apparent horizons.

It will be important to restrict the type of regular level sets that a spacetime harmonic function can have. In order to aid with this task, we will often pass from the given initial data to a \textit{generalized exterior region} associated with a particular end. More precisely, it is shown in Proposition \ref{exterior} below that for each end $M_{end}$, there exists a new initial data set $(M_{ext},g_{ext},k_{ext})$ having a single end which is isometric (as initial data) to the original end. In addition $M_{ext}$ is orientable, satisfies $H_2(M_{ext},\partial M_{ext};\mathbb{Z})=0$, and has a (possibly empty) boundary $\partial M_{ext}$ consisting entirely of MOTS and MITS.
Any asymptotically flat initial data set with one end satisfying these later three properties will be referred to as a generalized exterior region. Although $M_{ext}$ is not necessarily a subset of $M$, we will often denote the metric and extrinsic curvature of $M_{ext}$ by $(g,k)$ for convenience.
In the Riemannian setting where MOTS and MITS correspond to minimal surfaces, the (generalized) exterior region is a subset of the initial data which is diffeomorphic to the complement of a finite number of balls (with disjoint closure) in $\mathbb{R}^3$ \cite[Lemma 4.1]{HuiskenIlmanen}.

%By removing a sufficient number of MOTS and MITS, including those which are %immersed, it is possible to isolate the trapped region and obtain the %\textit{exterior region} $M_{ext} \supset M_{end}$ associated with an end, see %Proposition \ref{exterior} below. The exterior region has a boundary $\partial %M_{ext}$ consisting entirely of MOTS and MITS, and its topology is generally %simpler than that of $M$.
%In the Riemannian setting where MOTS and MITS correspond to minimal surfaces, the %exterior region is diffeomorphic to the complement of a finite number of balls %(with %disjoint closure) in $\mathbb{R}^3$ \cite[Lemma 4.1]{HuiskenIlmanen}.
%In particular $H_2(M_{ext},\partial M_{ext})=0$, and it is precisely this property %which is needed to effectively deal with the level sets of spacetime harmonic %functions. An asymptotically flat initial data set with one end satisfying this %topological condition, in addition to having only MOTS and MITS boundary %components, will be called an exterior region.

A function $u:M\rightarrow\mathbb{R}$ will be referred to as a \textit{spacetime harmonic function} if it satisfies the equation
\begin{equation}
\Delta u+\left(\mathrm{Tr}_g k\right)|\nabla u|=0,
\end{equation}
where $\Delta=g^{ij}\nabla_{ij}$ denotes the Laplace-Beltrami operator associated with the metric $g$. As discussed in Section \ref{sec2.5}, this equation may be interpreted as the trace along $M$ of the spacetime Hessian
\begin{equation}
\bar{\nabla}_{ij}u=\nabla_{ij}u+k_{ij}|\nabla u|,
\end{equation}
and therefore has similarities with the hypersurface Dirac operator \cite{ParkerTaubes,Witten} induced on the initial data from spacetime.  Furthermore, in analogy with spacetime harmonic spinors, such functions satisfy a Lichnerowicz-type integral identity whose boundary terms can be employed to extract the ADM quantities. A spacetime harmonic function $u$ defined on a generalized exterior region $M_{ext}$ will be called \textit{admissible} if it has constant Dirichlet boundary data, and satisfies $\partial_{\upsilon}u\leq (\geq) 0$ on each boundary component with $\theta_+ =0$ ($\theta_- =0$), where $\upsilon$ is the unit boundary normal pointing outside $M_{ext}$. Such spacetime harmonic functions asymptotic to any given linear function in $M_{end}$ always exist, see Theorem \ref{Thm:PDE} and Lemma \ref{l:dirichlet}. The main result is the following lower bound for the difference of ADM energy and linear momentum.

\begin{theorem}\label{main}
Let $(M_{ext},g,k)$ be a 3-dimensional generalized exterior region which is complete as a manifold with (possibly empty) boundary,
%of a 3-dimensional complete asymptotically flat initial data set $(M,g,k)$ for the %Einstein equations,
and has ADM energy $E$ and linear momentum $P$. Let $u$ be an admissible spacetime harmonic function on $M_{ext}$,  which is asymptotic to a linear combination $\langle \vec{a},x\rangle =a_i x^i$ of asymptotically flat coordinates of the associated end, with $|\vec{a}|=1$. Then
\begin{equation}\label{mainlowerb}
E+\langle \vec{a}, P\rangle\geq \frac{1}{16\pi} \int_{M_{ext}}\left(\frac{|\bar{\nabla}^2 u|^2}{|\nabla u|}
+2(\mu-|J|_g)|\nabla u|\right) dV.
\end{equation}
In particular, if the dominant energy condition holds then $E\geq |P|$. Furthermore, if $E=|P|$ then $E=|P|=0$, $M_{ext}$ is diffeomorphic to $\mathbb{R}^3$, and the data arise from an isometric embedding into Minkowski space.
\end{theorem}

A version of Theorem \ref{main} also holds if $(M_{ext},g,k)$ has weakly trapped boundary. That is, each boundary component satisfies $\theta_{+}\leq 0$ or $\theta_{-}\leq 0$, where the mean curvature is computed with respect to the unit normal pointing inside of $M_{ext}$. Although, in this situation the conclusion is a strict inequality $E>|P|$. Theorem \ref{main} generalizes a previous result \cite{BrayKazarasKhuriStern} in the Riemannian case when $k=0$. However, it is important to note that the boundary conditions for the harmonic functions in \cite{BrayKazarasKhuriStern} are homogeneous Neumann, as opposed to the constant Dirichlet utilized here. Thus, the lower bound of Theorem \ref{main} does not generally reduce to that of Theorem 1.2 in \cite{BrayKazarasKhuriStern} when $k=0$. In particular, we obtain an independent proof of the Riemannian results with the new boundary conditions.

The proof has some similarity with the stable MOTS approach of Eichmair, Huang, Lee, and Schoen \cite{EichmairHuangLeeSchoen} where regular level sets of the spacetime harmonic functions reprise the role of stable MOTS. The closest analogy though is to the spinorial method of Witten \cite{Witten}, where the spacetime harmonic functions play the role of spacetime harmonic spinors. While the spinor proof of the positive mass theorem also yields a lower bound for the mass, it should be noted that this is predicated on the dominant energy condition assumption. More precisely, the existence result for spacetime harmonic spinors converging to a given constant spinor in the asymptotic end relies on the dominant energy condition. On the other hand, the existence of an admissible spacetime harmonic function asymptotic to a given linear function in the asymptotic end is guaranteed regardless of any local energy hypotheses. Hence, the lower bound \eqref{mainlowerb} remains valid under general conditions far from those under which the positive mass theorem is known to hold.

Generalized exterior regions always exist for each end of an asymptotically flat initial data set satisfying the dominant energy condition, see Proposition \ref{exterior}. Therefore Theorem \ref{main} may be applied to arbitrary asymptotically flat initial data to obtain a lower bound for the difference of energy and momentum. The rigidity statement in this situation is stronger, as it implies topological rigidity and an isometric embedding into Minkowski space for the full initial data, as opposed to an isometric embedding only for the generalized exterior region. It should be pointed out that this isometric embedding may be given explicitly as a graph over a time slice in Minkowski space via a linear combination of three spacetime harmonic functions. This leads to a new and relatively simple proof of the positive mass theorem, Theorem \ref{positivemass}.

\begin{cor}\label{main1}
Let $(M,g,k)$ be a 3-dimensional complete asymptotically flat initial data set for the Einstein equations, having ADM energy $E$ and linear momentum $P$ in a chosen asymptotic end $M_{end}$. If the dominant energy condition is satisfied, then there exists a generalized exterior region $M_{ext}$ associated with $M_{end}$, which also satisfies the dominant energy condition.
Let $u$ be an admissible spacetime harmonic function on $M_{ext}$, which is asymptotic to a linear combination $\langle \vec{a},x\rangle =a_i x^i$ of asymptotically flat coordinates of the associated end, with $|\vec{a}|=1$. Then
\begin{equation}\label{mainlowerb5}
E+\langle \vec{a}, P\rangle\geq \frac{1}{16\pi} \int_{M_{ext}}\left(\frac{|\bar{\nabla}^2 u|^2}{|\nabla u|}
+2(\mu-|J|_{g})|\nabla u|\right) dV.
\end{equation}
In particular $E\geq |P|$. If $E=|P|$ then $E=|P|=0$, $M$ is diffeomorphic to $\mathbb{R}^3$, and the data $(M,g,k)$ arise from an isometric embedding into Minkowski space.
\end{cor}

\section{Generalized Exterior Regions}
\label{sec2} \setcounter{equation}{0}
\setcounter{section}{2}

A Lichnerowicz-type integral identity for spacetime harmonic functions lies at the core of the proof of Theorem \ref{main} and Corollary \ref{main1}. It is applied to generalized exterior regions, where one may control the topology of level sets. However, the integral formula holds in much greater generality. In this section we show that generalized exterior regions always exist, and in the next section we establish the desired identity.
%\subsection{Existence of generalized exterior regions}

In Lemma 4.1 of \cite{HuiskenIlmanen} Huisken and Ilmanen established the existence of an exterior region for asymptotically flat Riemannian 3-manifolds, showing that for each asymptotic end there is such a region which is diffeomorphic to the complement of a finite union of balls in $\mathbb{R}^3$. They accomplished this by removing all compact minimal surfaces, including immersed ones, to identify the trapped region and remove it. As pointed out by Lee in \cite[page 140]{Lee}, the weaker topological simplification $H_2(M_{ext},\partial M_{ext};\mathbb{Z})=0$ may still be achieved by only removing embedded compact minimal surfaces. His proof relies on the classical result that within each nontrivial 2-dimensional homology class there exists an area minimizing minimal surface representative.
Due to the lack of a variational characterization, such a result is not currently known for MOTS. Nevertheless, the conclusion of Lee's observation still remains valid in spirit with the role of minimal surfaces replaced by that of MOTS and MITS.

\begin{prop}\label{exterior}
Let $(M,g,k)$ be a smooth asymptotically flat initial data set satisfying the dominant energy condition. Then for each end $M_{end}$, there exists a new initial data set $(M_{ext},g_{ext},k_{ext})$ having a single end which is isometric (as initial data) to $(M_{end},g,k)$. Furthermore, $M_{ext}$ is orientable, satisfies $H_2(M_{ext},\partial M_{ext};\mathbb{Z})=0$, and has a boundary
$\partial M_{ext}$ consisting entirely of MOTS and MITS.
\end{prop}

\begin{proof}
There are two primary steps. The first is to identify appropriate (possibly immersed) MOTS and MITS to remove from $M$ in order to obtain a subset $M'\supset M_{end}$, whose compactification admits a positive scalar curvature metric. The second step entails reducing the first Betti number of $M'$ to zero via an iterative process which involves passing to finite sheeted covers. The proof of the first step is based on a reorganization of the arguments used for \cite[Theorem 1.2]{AnderssonDahlGallowayPollack}, and thus only an outline of the main ideas will be given here. The second step will be described in detail. In what follows, we assume without loss of generality that $M$ is orientable by passing to the orientable double cover if necessary.

According to \cite[Theorem 22]{EichmairHuangLeeSchoen} there is a sequence of perturbed initial data $(M,g_i,k_i)$ with $g_i \rightarrow g$ in $W^{3,p}_{-q}(M)$ and $k_i \rightarrow k$ in $W^{2,p}_{-q-1}(M)$ as $i\rightarrow\infty$ for $p>3$, such that a strict dominant energy condition is satisfied $\mu_i> |J_i|_{g_i}$.
To this end, solve the Jang equation \cite[Proposition 7]{Eichmair} for $(M,g_i,k_i)$ with standard asymptotic decay in each end. Note that the assumed decay on $\mathrm{Tr}_g k$ is not in general sufficient to guarantee bounded solutions of Jang's equation near infinity. However, as pointed out in \cite[Remarks 2.2 and 3.1]{AnderssonDahlGallowayPollack}, this technicality can be avoided by an appropriate deformation of the initial data in the asymptotic ends.
The solution of Jang's equation gives rise to a hypersurface in $\mathbb{R}\times M$ which is a vertical graph over an open subset of $M$ containing the asymptotic ends; $\Omega_i\subset M$ will denote the component of this open set that contains the designated end $M_{end}$. The components of the boundary $\partial\Omega_i$ are spherical MOTS or MITS that satisfy a uniform $C$-almost minimization property \cite[Remark 2.3]{AnderssonDahlGallowayPollack}, \cite{Eichmair-1}. Note that the spherical
topology is due to the strict dominant energy condition and stability property of the Jang graph. Observe that due to the strict dominant energy condition, the proof of \cite[Theorem 1.2]{AnderssonDahlGallowayPollack} shows that a conformal change of metric may be introduced, after preliminary deformations along the asymptotically cylindrical ends as well as in the asymptotically flat ends, to arrive at a positive scalar curvature (PSC) metric on the manifold obtained by compactifying the asymptotically flat ends of $\Omega_i$, which also has a Riemannian product structure near each boundary component.

Next, by the compactness theory of \cite{Eichmair-1,Eichmair0}, the sequence $\partial\Omega_i$ subconverges in the $C^{2,\alpha}$ local graph sense to a set $\mathcal{S}$ which is a finite collection of MOTS $\{\mathcal{S}^+_{a}\}_{a=1}^{a_0}$ and MITS $\{\mathcal{S}^-_{b}\}_{b=1}^{b_0}$ in $(M,g,k)$. Moreover, each of these MOTS and MITS arises from a sequence of connected closed properly embedded MOTS $\mathcal{S}^+_{ai}\subset \partial\Omega_i$ or MITS $\mathcal{S}^-_{bi}\subset \partial\Omega_i$ with respect to $(g_i,k_i)$.
We claim that $\mathcal{S}$ is a smooth submanifold. If a MOTS $\mathcal{S}^+_a$ or a MITS $\mathcal{S}^-_b$ remains disjoint from the other MOTS and MITS of $\mathcal{S}$, then this component is a smooth submanifold.
If $\mathcal{S}^+_{a}$ or $\mathcal{S}^-_b$ has nontrivial intersection and does not coincide with another member of the MOTS and MITS comprising $\mathcal{S}$,
this violates the $C$-almost minimization property of $\partial\Omega_i$ for large $i$. Thus the MOTS and MITS in $\mathcal{S}$ are pairwise disjoint, and hence
are smooth submanifolds.

To conclude the first step, remove the surface $\mathcal{S}$ from $M$ and take the metric completion of the component containing the designated end $M_{end}$ to obtain an initial data set $(M',g,k)$. Note that this contains $(M_{end},g,k)$, has boundary components consisting entirely of smooth MOTS and MITS, and the topology of $M'$ agrees with that of $\Omega_i$ for large $i$. Because $\Omega_i$ admits a PSC metric having Riemannian product structure near each boundary component, we may apply the prime decomposition theorem along with a result of Gromov-Lawson \cite{GromovLawson} and the resolution of the Poincar\'{e} conjecture to deduce that manifold $M'$ has PSC topology. That is, $M'$ is diffeomorphic to a finite connected sum of spherical spaces, $S^1\times S^2$'s, and $\mathbb{R}^3$'s representing the ends, all with a finite number of 3-balls removed which indicate the horizons. Thus, to conclude the first step of the proof, we have produced an asymptotically flat initial data set $(M',g,k)$ having PSC topology, with boundary $\partial M'$ consisting of MOTS and MITS components, and is such that one of the ends coincides with $(M_{end},g,k)$.

%Let $(M,g,k)$ be a $3$-dimensional asymptotically flat initial data set satisfying the dominant energy condition with a single end and boundary $\partial M=\partial^+M\cup\partial^-M$ where $\partial^\pm M$ satisfies $\theta^\pm=0$. Assume that the underlying manifold $M$ is diffeomorphic to the compliment of finitely many disjoint balls in $\#^rS^1\times S^2\#N$ where $N$ is a rational homology sphere. Then there exists a generalized exterior region $(M',g',k')$ for $(M,g,k)$ such that $H_2(M',\partial M';\mathbb{Z})$ is trivial and each component of $\partial M'$ is either a MOTS or MITS.

In the second step of the proof the first Betti number of $M'$ will be reduced to zero with an iterative procedure. Since $H_2(M',\partial M';\mathbb{Z})$ is Poincar{\'{e}} dual to $H^1(M';\mathbb{Z})$, which is itself isomorphic to the torsion-free subgroup of $H_1(M';\mathbb{Z})$, this procedure will result in the desired conclusion of vanishing second homology relative to the boundary.
As observed above, $M'$ can be expressed as the compliment of finitely many disjoint balls in $\#^l (S^1\times S^2)\# N$ where $N$ is a rational homology sphere. Since $N$ has vanishing first Betti number, $b_1(M')$ is equal to the number of its handle $S^1\times S^2$ summands. We proceed by constructing a particular double cover of $M'$. Let $\Sigma'\subset \mathring{M}'$ be the image of an embedding of $S^2$ in one of the $S^1\times S^2$ summands of $M'$ which is homologous $\{\mathrm{pt}\}\times S^2\subset S^1\times S^2$. Define $W$ to be the metric completion of $M'\setminus\Sigma'$ and notice that its boundary can be decomposed as
\begin{equation}
\partial W=\partial M'\cup\Sigma'_1\cup\Sigma'_2,
\end{equation}
where $\Sigma'_1$ and $\Sigma'_2$ are copies of $\Sigma'$. Next, consider the manifold
\begin{equation}
\overline{M}=W_1\sqcup W_2/\sim,
\end{equation}
where $W_1$ and $W_2$ are copies of $W$ and the relation $\sim$ identifies $\Sigma'_1\subset W_1$ with $\Sigma'_2\subset W_2$ and $\Sigma'_2\subset W_1$ with $\Sigma'_1\subset W_2$. The manifold $\overline{M}$ is a two-fold cover of $M$, classified by the mod $2$ reduction of the cohomology class Poincar{\'e} dual to $[\Sigma]$, and the pullback of the data $(g,k)$ to $\overline{M}$ will be denoted by $(\overline{g},\overline{k})$. Furthermore, observe that $\overline{M}$ is diffeomorphic to the complement of finitely many disjoint balls in
\begin{equation}
\left(\#^{l-1}(S^1\times S^2)\# N\right)\# (S^1\times S^2)\#\left(\#^{l-1}(S^1\times S^2)\# N\right),
\end{equation}
so that
\begin{equation}\label{e:betti}
b_1(\overline{M})=2b_1(M')-1.
\end{equation}

Consider the two ends of $\overline{M}$ that are isometric to $M_{end}$, and choose one for reference and denote it by $\mathcal{E}$. The boundary of the double cover may be decomposed as $\partial\overline{M}=\partial_+ \overline{M}\cup\partial_-\overline{M}$, where $\theta_\pm=0$ on $\partial_\pm \overline{M}$ and the null expansions are computed with respect to the unit normal pointing inside $\overline{M}$. Now let $\mathcal{D}\subset \overline{M}$ be the bounded component that remains after removing sufficiently large coordinate spheres in each of the asymptotic ends of $\overline{M}$. The boundary may be decomposed into two types of surfaces $\partial\mathcal{D}=\partial_{out}\mathcal{D}\cup\partial_{in}\mathcal{D}$, in which $\theta_+\geq 0$ on $\partial_{out}\mathcal{D}$ with respect to the normal pointing out of $\mathcal{D}$, and $\theta_+\leq 0$ on $\partial_{in}\mathcal{D}$ with respect to the normal pointing into $\mathcal{D}$. Note that MOTS boundary components belong to $\partial_{in}\mathcal{D}$, while MITS components belong to $\partial_{out}\mathcal{D}$. Moreover the coordinate sphere boundary in $\mathcal{E}$ satisfies the strict inequality $\theta_+>0$ and belongs to $\partial_{out}\mathcal{D}$, while the coordinate sphere boundaries lying in the remaining ends satisfy the strict inequality $\theta_+<0$ and belong to $\partial_{in}\mathcal{D}$. It follows that we may apply the MOTS existence result \cite[Theorem 4.2]{EichmairGallowayPollack}, or rather a slight generalization of it to allow for nonstrict inequalities (see \cite[Section 5]{AnderssonMetzger} or \cite[Remark 4.1]{Eichmair0}), to obtain an outermost (with respect to $\mathcal{E}$) MOTS  $\Sigma\subset \mathcal{D}$ that separates $\partial_{out}\mathcal{D}$ from $\partial_{in}\mathcal{D}$. Furthermore, this surface separates $\overline{M}$ into two disjoint regions $\overline{M}\setminus \Sigma=\overline{M}_{out}\cup\overline{M}_{in}$, where $\overline{M}_{out}$ is the component containing the reference end $\mathcal{E}$, see Figure \ref{pic:section2}.

In the remainder of the argument, we will first consider the case in which $(\overline{M},\overline{g},\overline{k})$ satisfies a strict dominant energy condition, and will subsequently explain the alterations required for the general case. By the strict dominant energy condition, stability of outermost MOTS, and orientability of $\overline{M}$, it follows that $\Sigma$ consists of finitely many disjoint embedded spheres. Now consider the Mayer-Vietoris sequence associated with the decomposition $\overline{M}=\overline{M}_{out}\cup\overline{M}_{in}$, that is
\begin{equation}
\begin{tikzcd}
\cdots\arrow[r]&
H_1(\Sigma;\mathbb{R})\arrow[r]&
H_1(\overline{M}_{out};\mathbb{R})\oplus H_1(\overline{M}_{in};\mathbb{R})\arrow[r]&
H_1(\overline{M};\mathbb{R})\arrow[r]&\cdots.
\end{tikzcd}
\end{equation}
Since $H_1(\Sigma;\mathbb{R})=0$ we find that
that $b_1(\overline{M}_{out})+b_1(\overline{M}_{in})\leq b_1(\overline{M})$. Taking \eqref{e:betti} into consideration shows that either $\overline{M}_{out}$ or $\overline{M}_{in}$ must have first Betti number strictly less than $b_1(M')$; label the component of this manifold that contains an isometric copy of $M_{end}$, by $\overline{M}'$. Notice that each component of the boundary of $\overline{M}'$ is either a MOTS or a MITS. Moreover, as $\Sigma$ is spherical, the sets $\overline{M}_{out}$ and $\overline{M}_{in}$ give rise to a connected sum decomposition of $\overline{M}$. It follows that both $\overline{M}_{out}$ and $\overline{M}_{in}$ are diffeomorphic to the compliment of finitely many disjoint balls in the connected sum of $S^1\times S^2$'s and a rational homology sphere. Furthermore, we may assume that $\overline{M}'$ has a single end, since if necessary attention may be restricted to the region outside the outermost MOTS to isolate the isometric copy of $M_{end}$.
This same procedure can be applied to $\overline{M}'$ to once again reduce the first Betti number by at least one. Continuing in this manner yields the desired initial data $(M_{ext},g_{ext},k_{ext})$.

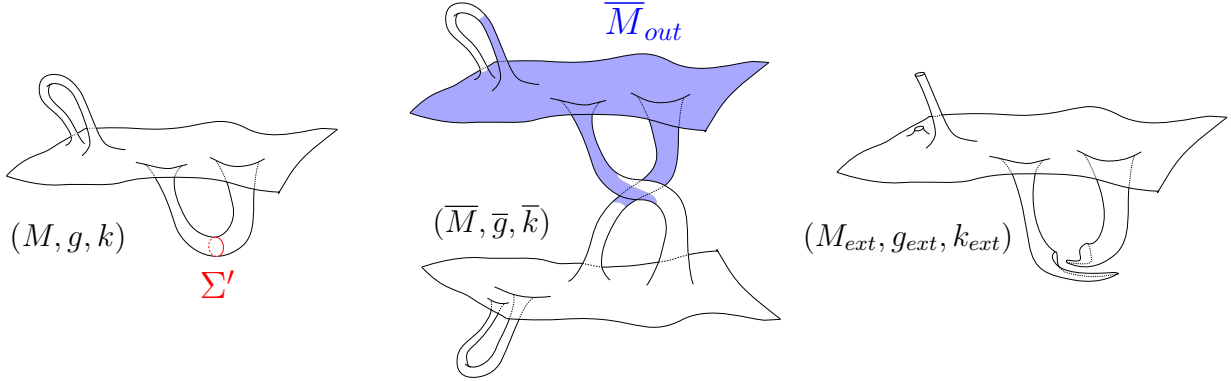
\begin{figure}
\begin{picture}(-10,0)
\put(-5,50){\large{$(M,g,k)$}}
\put(67,30){{\color{red}\Large{$\Sigma'$}}}
\put(155,55){\large{$(\overline{M},\overline{g},\overline{k})$}}
\put(220,130){{\color{blue}\Large{$\overline{M}_{out}$}}}
\put(295,50){\large{$(M_{ext},g_{ext},k_{ext})$}}
\end{picture}
\includegraphics[scale=.4]{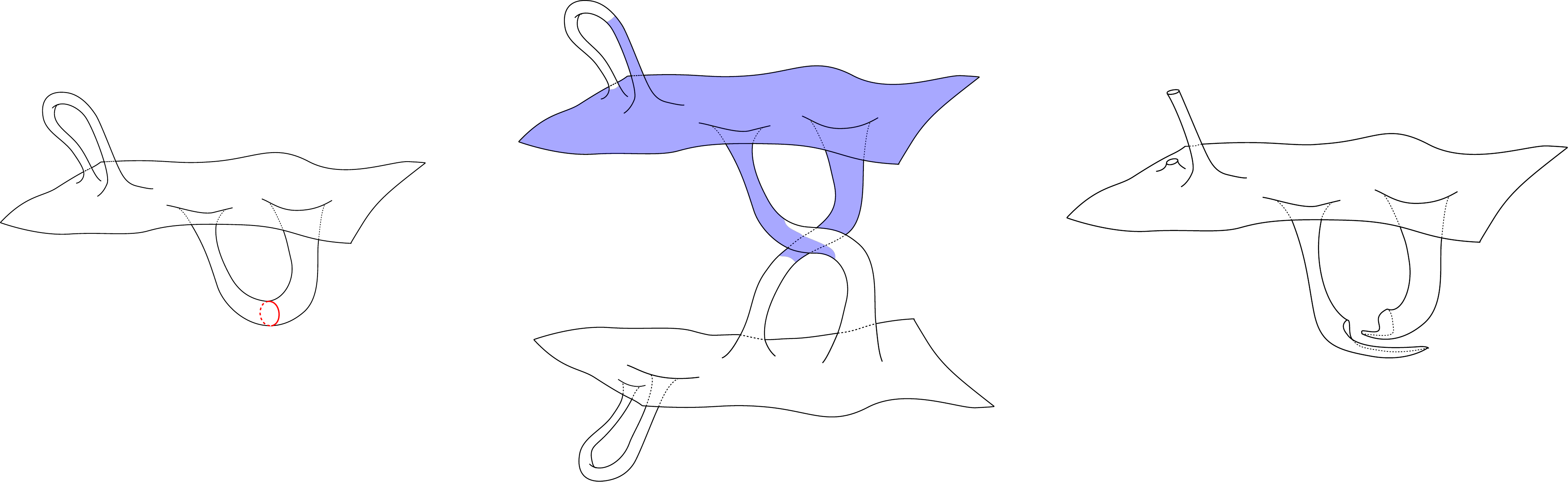}
\caption{A schematic description of the stages in the second step in the proof of Proposition \ref{exterior}.}\label{pic:section2}
\end{figure}

To finish, we describe the modifications necessary to accomplish the construction in the above paragraph in the general case when $(\overline{M},\overline{g},\overline{k})$ satisfies the dominant energy condition, but not strictly so. In this case, apply the approximating argument from the first step to obtain a sequence $(\overline{g}_i,\overline{k}_i)$ on $\overline{M}$ satisfying the strict dominant energy condition and which converges to $(\overline{g},\overline{k})$. Note that a minor refinement of \cite[Theorem 22]{EichmairHuangLeeSchoen} is required for this due to the presence of boundary components, see \cite[footnote - page 869]{AnderssonDahlGallowayPollack}. The outermost MOTS $\Sigma_i$ that induces a separation $\overline{M}=\overline{M}_{out}^i \cup \overline{M}_{in}^i$, admits the $C$-almost minimization property and consists of spherical MOTS and MITS.
By the arguments of the previous paragraph, the first Betti number of either $\overline{M}_{out}^i$ or $\overline{M}_{in}^i$ is strictly less than $b_1(M')$. As described in the first step of the proof, $\Sigma_i$ subconverges to a limiting MOTS/MITS surface $\overline{\mathcal{S}}$ in $\overline{M}$, and we may consider the metric completion $\widehat{M}$ of $\overline{M}\setminus\overline{\mathcal{S}}$. The two components of $\widehat{M}$ containing the isometric copies of $M_{end}$, have the
same topology as components of $\overline{M}_{out}^i$ or $\overline{M}_{in}^i$ for sufficiently large $i$. It follows that one of them, $\widehat{M}'$, satisfies $b_1(\widehat{M}')<b_1(M')$. As above it may be assumed that the component $\widehat{M}'$ possesses one end modeling $\mathcal{E}$. Moreover its boundary consists of MOTS and MITS, and it is diffeomorphic to the compliment of finitely many disjoint balls in the connected sum of $S^1\times S^2$'s and a rational homology sphere. Thus the iteratation may be continued to obtain the desired conclusion.
\end{proof}

\section{The Integral Formula}
\label{sec2.5} \setcounter{equation}{0}
\setcounter{section}{3}
%\subsection{The integral formula for spacetime harmonic functions}

Let $(\widetilde{M}^4,\widetilde{g})$ be an asymptotically flat spacetime with initial data slice $(M,g,k)$. In the time-symmetric case when $k=0$ the positive mass theorem was proven \cite{BrayKazarasKhuriStern} using asymptotically linear harmonic functions on $(M,g)$.
In order to find an analogue in the spacetime setting, it is helpful to obtain intuition
from the case of equality. If the mass vanishes, $m=0$, then the spacetime in which the initial data set resides should be Minkowski space $\mathbb{M}^4$. In the time symmetric case the initial data coincide with a constant time slice, and the harmonic functions used are simply the linear functions of the coordinates of this slice, that is $a_i x^i$ where $a_i$, $i=1,2,3$ are constants. For nonconstant time slices, a natural generalization would be to use linear functions of the coordinates in $\mathbb{M}^4$, that is $a_0 x^0+a_i x^i$, and restrict this function to the slice $(M,g,k)$. It then remains to find a canonical equation induced on the slice which is satisfied by these functions.

To this end, let $\widetilde{\nabla}$ and $\nabla$ denote the Levi-Civita connections of the spacetime and slice, respectively. Since the Minkowski coordinate functions have vanishing spacetime Hessian, their linear combinations restricted to the slice must be in the kernel of the
hypersurface spacetime Laplacian $\widetilde{\Delta}=g^{ij}\widetilde{\nabla}_{ij}$. In a general spacetime, we are then motivated to consider functions $\widetilde{u}\in C^{\infty}(\widetilde{M}^4)$ which satisfy
\begin{equation}\label{eq:prespacetimehess}
0=\widetilde{\Delta}\widetilde{u}=g^{ij}\left(\nabla_{ij}\widetilde{u}-k_{ij} \mathbf{n}(\widetilde{u})\right)
=\Delta\widetilde{u}-\left(\mathrm{Tr}_g k\right)\mathbf{n}(\widetilde{u})\quad\text{ on }\quad M,
\end{equation}
where $\mathbf{n}$ is the unit timelike normal to the slice.
Notice that $\widetilde{\Delta}\widetilde{u}$ can be considered as the divergence of $\widetilde{\nabla}u|_M$
using the spacetime connection $\widetilde{\nabla}$ acting on sections of the induced bundle $T\widetilde{M}|_{M}$.
This is in direct analogy to the spinorial proof of the positive mass theorem where Witten \cite{Witten} considers a
Dirac operator defined using a connection induced by $\widetilde{M}$. Equation \eqref{eq:prespacetimehess}, however, does not depend solely on the restriction $u=\widetilde{u}|_{M}$ due to the presence of the normal derivative. A choice for $\mathbf{n}(\widetilde{u})$ must then be made in order to obtain a purely intrinsic equation on the slice. It turns out that the desired choice for our purposes is to choose the normal derivative so that the spacetime gradient of $\widetilde{u}$ is null, that is $\mathbf{n}(\widetilde{u})=-|\nabla u|$.

With the above discussion in mind, we make the following definitions. Given initial data $(M,g,k)$ and a function $u\in C^2(M)$, set
\begin{equation}\label{spacetimehess}
\bar{\nabla}^2u=\nabla^2u+k|\nabla u|.
\end{equation}
A function $u\in C^{2}(M)$ is {\emph{spacetime harmonic}} with respect to initial data $(M,g,k)$ if $\bar{\Delta}u=0$ where
\begin{equation}\label{shf}
\bar{\Delta}u=\mathrm{Tr}_g\bar{\nabla}^2u=\Delta u+\mathcal{K}|\nabla u|.
\end{equation}
In the above, we have made use of the shorthand $\mathcal{K}=\mathrm{Tr}_gk$. Note that \eqref{spacetimehess} agrees with the spatial components of the spacetime Hessian when the normal derivative off of the slice is chosen as above.

Before stating the primary integral formula for spacetime harmonic functions, we give a technical lemma based on a refined version of Kato's inequality. This will be used in the proof of the main result of this section. Note that the natural regularity for spacetime harmonic functions is $C^{2,\alpha}(M)$, $0<\alpha<1$. By Rademacher's theorem $|\nabla u|$ is then differentiable almost everywhere, and from the equation the same holds for $\Delta u$. Thus, the inequality of the next result holds away from a set of measure zero.

\begin{lemma}\label{kato}
Let $u$ be a spacetime harmonic function for the initial data set $(M,g,k)$. Then there exists a constant $C>0$ depending only on $\mathcal{K}$ and its first derivatives such that
\begin{equation}\label{890}
|\nabla^2 u|^2 -|\nabla|\nabla u||^2 +\langle\nabla u,\nabla\Delta u\rangle \geq
-C|\nabla u|^2.
\end{equation}
\end{lemma}

\begin{proof}
By using the spacetime harmonic function equation \eqref{shf} we have
\begin{equation}\label{-=}
\langle\nabla u,\nabla\Delta u\rangle \geq
-\mathcal{K}\langle\nabla u,\nabla |\nabla u|\rangle
-C_0 |\nabla u|^2\geq -\frac{1}{4}|\nabla|\nabla u||^2 -C_1 |\nabla u|^2.
\end{equation}
Moreover, a refined version of the Kato inequality produces
\begin{equation}\label{rk}
|\nabla^2 u|^2\geq \frac{5}{4}|\nabla|\nabla u||^2 -C_2 |\nabla u|^2.
\end{equation}
Note that as discussed above, these inequalities hold almost everywhere. Combining \eqref{-=} and \eqref{rk} yields the desired result.

It remains to establish \eqref{rk}. To this end denote $u_i =\partial_i u$ and set
\begin{equation}
X_i=\frac{1}{2}\partial_i |\nabla u|^2-\frac{1}{3}(\Delta u)u_i,\quad\quad
W_{ij}=X_{(i}u_{j)}-\frac{1}{3}\langle X,\nabla u\rangle g_{ij},
\end{equation}
where parentheses are used to indicate symmetrization of indices. Observe that
\begin{align}
\begin{split}
|W|^2=&X^i u^j \left( X_{(i}u_{j)}-\frac{1}{3}\langle X,\nabla u\rangle g_{ij}\right)\\
=&\frac{1}{2}|X|^2 |\nabla u|^2 +\frac{1}{6}\langle X,\nabla u\rangle^2\\
\leq & \frac{2}{3}|X|^2 |\nabla u|^2,
\end{split}
\end{align}
which implies that
\begin{align}
\begin{split}
\frac{1}{2}X^i \partial_i |\nabla u|^2 =& X^i u^j \nabla_{ij}u\\
=&X^i u^j \left(\nabla_{ij}u -\frac{1}{3}(\Delta u)g_{ij}\right)+\frac{1}{3}(\Delta u)\langle X,\nabla u\rangle\\
=&W^{ij} \left(\nabla_{ij}u -\frac{1}{3}(\Delta u)g_{ij}\right)+\frac{1}{3}(\Delta u)\langle X,\nabla u\rangle\\
\leq&|W|\sqrt{|\nabla^2 u|^2-\frac{1}{3}(\Delta u)^2}+\frac{1}{3}(\Delta u)\langle X,\nabla u\rangle\\
\leq &\sqrt{\frac{2}{3}}|X||\nabla u|\sqrt{|\nabla^2 u|^2-\frac{1}{3}(\Delta u)^2}+\frac{1}{3}(\Delta u)\langle X,\nabla u\rangle.
\end{split}
\end{align}
It follows that
\begin{equation}
|X|\leq\sqrt{\frac{2}{3}}|\nabla u|\sqrt{|\nabla^2 u|^2-\frac{1}{3}(\Delta u)^2}.
\end{equation}
Squaring both sides, utilizing the spacetime harmonic function equation, and applying Young's inequality then gives
\begin{align}
\begin{split}
|\nabla u|^2|\nabla^2 u|^2\geq &\frac{1}{3}(\Delta u)^2|\nabla u|^2+\frac{3}{2}|X|^2\\
=&\frac{1}{2}(\Delta u)^2 |\nabla u|^2+\frac{3}{2}|\nabla|\nabla u||^2 |\nabla u|^2
-(\Delta u)|\nabla u|\langle\nabla u,\nabla|\nabla u|\rangle\\
\geq&\frac{5}{4}|\nabla|\nabla u||^2|\nabla u|^2-C_2|\nabla u|^4.
\end{split}
\end{align}
This gives inequality \eqref{rk}, if $|\nabla u|\neq 0$. At points where $|\nabla u|=0$ and $|\nabla u|$ is differentiable, we have that $|\nabla |\nabla u||=0$ since the nonnegative function $|\nabla u|$ achieves its minimum value. Inequality \eqref{rk} thus holds trivially at such points. The remaining points, where $|\nabla u|=0$ and $|\nabla u|$ is not differentiable, form a set of measure zero.
\end{proof}

We are now in a position to establish the main integral formula for spacetime harmonic functions. This may be viewed as a generalization of \cite[Proposition 4.2]{BrayKazarasKhuriStern}, see also \cite{BrayStern,Stern}.

\begin{prop}\label{SpacetimeSternCompact}
Let $(\Omega,g,k)$ be a 3-dimensional oriented compact initial data set with smooth boundary $\partial\Omega$, having outward unit normal $\upsilon$. Let $u:\Omega\to\mathbb{R}$ be a spacetime harmonic function, and denote
the open subset of $\partial\Omega$ on which $|\nabla u|\neq 0$ by $\partial_{\neq 0}\Omega$. If $\overline{u}$ and $\underline{u}$ denote the maximum and minimum values of $u$ and $\Sigma_t$ are $t$-level sets, then
\begin{equation}
\int_{\partial_{\neq 0} \Omega}\left(\partial_{\upsilon}|\nabla u|+k(\nabla u,\upsilon)\right)dA\geq\int_{\underline{u}}^{\overline{u}}\int_{\Sigma_t}
\left(\frac{1}{2}\frac{|\bar{\nabla}^2 u|^2}{|\nabla u|^2}+\mu+J(\nu)-K\right)dA dt,
\end{equation}
where $\nu=\tfrac{\nabla u}{|\nabla u|}$ and $K$ is the level set Gauss curvature.
\end{prop}

\begin{proof}
Recall Bochner's identity
\begin{equation}
\frac12\Delta|\nabla u|^2=|\nabla^2u|^2+\Ric(\nabla u,\nabla u)+\langle \nabla u,\nabla \Delta u\rangle.
\end{equation}
For $\varepsilon>0$ set $\varphi_\varepsilon=\sqrt{|\nabla u|^2+\varepsilon}$, and use Bochner's identity to find
\begin{align}\label{1}
\begin{split}
\Delta\varphi_\varepsilon=&\frac{\Delta|\nabla u|^2}{2\varphi_{\varepsilon}}-\frac{|\nabla|\nabla u|^2|^2}{4\varphi_{\varepsilon}^3}\\
\geq&\frac1{\varphi_\varepsilon}\left(|\nabla^2 u|^2-|\nabla |\nabla u||^2+\Ric(\nabla u,\nabla u)+\langle \nabla u,\nabla \Delta u\rangle\right).
\end{split}
\end{align}
On a regular level set $\Sigma$, the unit normal is $\nu=\tfrac{\nabla u}{|\nabla u|}$ and the second fundamental form is given by $II=\frac{\nabla^2_\Sigma u}{|\nabla u|}$,
where $\nabla^2_\Sigma u$ represents the Hessian of $u$ restricted to $T\Sigma\otimes T\Sigma$. We then have
\begin{equation}
|II|^2=|\nabla u|^{-2}\left(|\nabla^2u|^2-2|\nabla|\nabla u||^2+[\nabla^2u(\nu,\nu)]^2\right),
\end{equation}
and the mean curvature satisfies
\begin{equation}
|\nabla u| H=\Delta u - \nabla^2_{\nu}u.
\end{equation}
Furthermore by taking two traces of the Gauss equations
\begin{equation}
2\mathrm{Ric}(\nu,\nu)=R_g -R_{\Sigma}-|II|^2 +H^2,
\end{equation}
where $R_g$ and $R_{\Sigma}$ are scalar curvatures. Combining these formulas with \eqref{1} produces
\begin{align}\label{r1}
\begin{split}
\Delta \varphi_\varepsilon\geq&\frac1{\varphi_\varepsilon}\left(|\nabla^2u|^2
-|\nabla|\nabla u||^2+\langle\nabla u,\nabla\Delta u\rangle
+\frac{|\nabla u|^2}{2}\left(R_g -R_{\Sigma}
+H^2 -|II|^2\right)\right)\\
=&\frac1{2\varphi_\varepsilon}\left(|\nabla^2u|^2
+(R_g-R_\Sigma)|\nabla u|^2+2\langle \nabla u,\nabla\Delta u\rangle+(\Delta u)^2-2(\Delta u)\nabla_{\nu}^2 u\right).
\end{split}
\end{align}

Let us now replace the Hessian with the spacetime Hessian via the relation $\bar{\nabla}^2u =\nabla^2 u+k|\nabla u|$, and utilize the spacetime harmonic function equation $\Delta u=-\mathcal{K}|\nabla u|$ to find
\begin{align}
\begin{split}
\Delta\varphi_{\varepsilon}\geq &\frac{1}{2\varphi_{\varepsilon}}\left(
|\bar{\nabla}^2 u|^2 -2\langle k,\nabla^2 u\rangle|\nabla u|-|k|_g^2 |\nabla u|^2
+(R_g-R_{\Sigma})|\nabla u|^2\right.\\
&\left. -2\langle\nabla u,\nabla\mathcal{K}\rangle|\nabla u|-2\mathcal{K}\langle\nabla u,\nabla|\nabla u|\rangle
+\mathcal{K}^2|\nabla u|^2+2\mathcal{K}|\nabla u|\nabla_{\nu}^2 u\right).
\end{split}
\end{align}
Moreover noting that
\begin{equation}
\langle\nabla u,\nabla|\nabla u|\rangle=\frac{1}{2}\langle\nu,\nabla|\nabla u|^2\rangle =u^i\nabla_{i\nu}u=|\nabla u|\nabla_{\nu}^2 u,\quad\quad
2\mu=R_g+\mathcal{K}^2 -|k|_g^2,
\end{equation}
gives rise to the following inequality on a regular level set
\begin{equation}\label{jhg}
\Delta\varphi_{\varepsilon}\geq \frac{1}{2\varphi_{\varepsilon}}\left(
|\bar{\nabla}^2 u|^2+(2\mu-R_{\Sigma})|\nabla u|^2 -2\langle k,\nabla^2 u\rangle|\nabla u|
-2\langle\nabla u,\nabla\mathcal{K}\rangle|\nabla u|\right).
\end{equation}

Consider an open set $\mathcal{A}\subset [\underline{u},\overline{u}]$ containing the critical values of $u$, and let $\mathcal{B}\subset[\underline{u},\overline{u}]$ denote the complementary closed set. Then integrate by parts to obtain
\begin{equation}
\int_{\partial\Omega}\partial_\upsilon \varphi_{\varepsilon} dA=\int_{\Omega}\Delta \varphi_{\varepsilon} dV=\int_{u^{-1}(\mathcal{A})}\Delta \varphi_{\varepsilon} dV+\int_{u^{-1}(\mathcal{B})}\Delta \varphi_{\varepsilon} dV.
\end{equation}
According to Lemma \ref{kato} and \eqref{1} there is a positive constant $C_0$, depending only on $\mathrm{Ric}(g)$ along with $\mathcal{K}$ and its first derivatives, such that
\begin{equation}
\Delta \varphi_{\varepsilon}
\geq-C_0 |\nabla u|.
\end{equation}
An application of the coarea formula to $u:u^{-1}(\mathcal{A})\to\mathcal{A}$ then produces
\begin{equation}
-\int_{u^{-1}(\mathcal{A})}\Delta \varphi_{\varepsilon} dV\leq C_0\int_{u^{-1}(\mathcal{A})}|\nabla u|dV
=C_0 \int_{t\in\mathcal{A}}\mathcal{H}^2(\Sigma_t)dt,
\end{equation}
where $\mathcal{H}^2(\Sigma_t)$ is the 2-dimensional Hausdorff measure of the $t$-level set $\Sigma_t$. Next, apply the coarea formula to $u:u^{-1}(\mathcal{B})\to\mathcal{B}$ together with \eqref{jhg} to find
\begin{equation}
\int_{u^{-1}(\mathcal{B})}\Delta \varphi_{\varepsilon} dV
\geq\frac{1}{2}\int_{t\in\mathcal{B}}\int_{\Sigma_t}\frac{|\nabla u|}{\varphi_{\varepsilon}}\left[\frac{|\bar{\nabla}^2 u|^2}{|\nabla u|^2}+2\mu-R_{\Sigma_t}-\frac{2}{|\nabla u|}(
\langle k,\nabla^2 u\rangle
+\langle\nabla u,\nabla\mathcal{K}\rangle)\right]dAdt.
\end{equation}
Combining all this together produces
\begin{align}\label{ooo}
\begin{split}
\int_{\partial\Omega}\partial_\upsilon \varphi_{\varepsilon} dA
+C_0 \int_{t\in\mathcal{A}}\mathcal{H}^2(\Sigma_t)dt\geq&
\frac{1}{2}\int_{t\in\mathcal{B}}\int_{\Sigma_t}\frac{|\nabla u|}{\varphi_{\varepsilon}}\left(\frac{|\bar{\nabla}^2 u|^2}{|\nabla u|^2}+2\mu-R_{\Sigma_t}\right)dAdt\\
&-\int_{t\in\mathcal{B}}\int_{\Sigma_t}\varphi_{\varepsilon}^{-1}
\left(\langle k,\nabla^2 u\rangle
+\langle\nabla u,\nabla\mathcal{K}\rangle\right)dAdt.
\end{split}
\end{align}

On the set $u^{-1}(\mathcal{B})$, we have that $|\nabla u|$ is uniformly bounded from below. In addition, on $\partial_{\neq 0}\Omega$ it holds that
\begin{equation}
\partial_{\upsilon}\varphi_{\varepsilon}=\frac{|\nabla u|}{\varphi_{\varepsilon}}
\partial_{\upsilon}|\nabla u|\rightarrow \partial_{\upsilon}|\nabla u|\quad\text{ as }\quad\varepsilon\rightarrow 0.
\end{equation}
Therefore, the limit $\varepsilon\to 0$ may be taken in \eqref{ooo}, resulting in the same bulk expression except that $\varphi_{\varepsilon}$ is replaced by $|\nabla u|$, and with the boundary integral taken over the restricted set. Furthermore, by Sard's theorem (see Remark \ref{798} below) the measure $|\mathcal{A}|$ of $\mathcal{A}$ may be taken to be arbitrarily small. Since the map $t\mapsto \mathcal{H}^2(\Sigma_t)$ is integrable over $[\underline{u},\overline{u}]$ in light of the coarea formula, by then taking $|\mathcal{A}|\to 0$ we obtain
\begin{align}
\begin{split}
\int_{\partial_{\neq 0}\Omega}\partial_\upsilon |\nabla u|dA
\geq
\frac{1}{2}\int_{\underline{u}}^{\overline{u}}\int_{\Sigma_t}\left(\frac{|\bar{\nabla}^2 u|^2}{|\nabla u|^2}+2\mu-R_{\Sigma_t}\right)dAdt
-\int_{\Omega}
\left(\langle k,\nabla^2 u\rangle
+\langle\nabla u,\nabla\mathcal{K}\rangle\right)dV.
\end{split}
\end{align}
Lastly integrating parts
\begin{equation}
-\int_{\Omega}\langle k,\nabla^2 u\rangle dV
=-\int_{\Omega}k^{ij}\nabla_{ij} u dV
=\int_{\Omega} u^i\nabla^j k_{ij}-\int_{\partial \Omega}k(\nabla u,\upsilon)dA,
\end{equation}
and recalling that $J=\operatorname{div}_g(k-\mathcal{K}g)$ and $R_{\Sigma_t}=2K$, yields the desired result.
\end{proof}

\begin{remark}\label{798}
The classical statement of Sard's theorem in the current context requires $u\in C^{3}$, while spacetime harmonic functions typically only satisfy $u\in C^{2,\alpha}$, $0<\alpha<1$. Nevertheless, Sard's theorem still applies. To see this, observe that since $|\nabla u|$ is Lipschitz and hence in $L^p_{loc}$ for all $p$, elliptic regularity yields $u\in W^{2,p}_{loc}$. It follows from Kato's inequality that $|\nabla u|\in W^{1,p}_{loc}$, and therefore $u\in W^{3,p}_{loc}$. According to \cite[Theorem 5]{Figalli} the conclusion of Sard's theorem holds for such functions.
\end{remark}

\section{Existence and Uniqueness of Spacetime Harmonic Functions}
\label{sec3} \setcounter{equation}{0}
\setcounter{section}{4}

Let $(M,g,k)$ be an asymptotically flat initial data set with (possibly empty)
smooth boundary $\partial M$. The purpose of this section is to establish the
appropriate existence, uniqueness, and asymptotic properties of spacetime harmonic functions. As we will see, even though the spacetime harmonic function equation is nonlinear, the nonlinearity is sufficiently mild that it behaves similar to a linear equation without zeroth order term. For simplicity of discussion, it will be assumed here that $M$ possesses a single end, although the final result stated at the end of the section will be given in full generality. Let $a_i x^i$ be a linear function of the asymptotic coordinates in the end $M_{end}$, with $\sum_{i}a_i^2 =1$, and let $h\in C^{\infty}(\partial M)$. By slightly generalizing \cite[Theorem 3.1]{Bartnik} we may solve the asymptotically linear Dirichlet problem
\begin{equation}\label{vequation}
\Delta v=-\mathcal{K}\quad\quad\text{ on }\quad\quad M,
\end{equation}
\begin{equation}\label{vequationb}
v=0\quad\text{ on }\quad \partial M,\quad\quad\quad\quad v=a_i x^i +O_2(r^{1-q})\quad\text{ as }\quad r\rightarrow\infty,
\end{equation}
where $r=|x|$, $q$ is as in \eqref{asymflat}, and $O_2$ indicates in the usual way additional fall-off for each derivative taken up to order 2. Consider now the corresponding problem for the spacetime harmonic function equation
\begin{equation}\label{88}
\Delta u+\mathcal{K}|\nabla u|=0\quad\quad\text{ on }\quad\quad  M,
\end{equation}
\begin{equation}\label{881}
u=h\quad\text{ on }\quad \partial M,\quad\quad\quad\quad u=v +O_2(r^{-\beta})\quad\text{ as }\quad r\rightarrow\infty,
\end{equation}
where $\beta\in (0,1)$. The strategy will be to first solve for $u$ on a sequence of compact domains exhausting $M$, use a barrier in the asymptotic end to obtain uniform estimates, and then find a subsequence that converges to the desired solution.

\subsection{Solutions on compact exhausting domains}\label{mr}

Let $S_r\subset M_{end}$ be a coordinate sphere in the asymptotic end, and let
$M_r$ denote the compact component of $M\setminus S_r$ having boundary $\partial M_r=\partial M \cup S_r$. Consider now the preliminary Dirichlet problem
\begin{equation}\label{ty}
\Delta u^r+\mathcal{K}|\nabla u^r|=0\quad\quad\text{ on }\quad\quad  M_r,
\end{equation}
\begin{equation}\label{ty0}
u^r=h\quad\text{ on }\quad \partial M,\quad\quad\quad\quad u^r=v \quad\text{ on }\quad S_r.
\end{equation}
For this boundary value problem we will use
the Leray-Schauder fixed point theorem \cite[Theorem 11.3]{GilbargTrudinger}.

\begin{theorem}
Let $\mathcal{B}$ be a Banach space and $\mathcal{F}:\mathcal{B}\times [0,1]\rightarrow \mathcal{B}$ a compact mapping with
$\mathcal{F}(b,0)=0$ for all $b\in \mathcal{B}$. If there is a constant $c$, such that any solution $(b,\sigma)\in \mathcal{B}\times[0,1]$ of $b=\mathcal{F}(b,\sigma)$ satisfies the a priori inequality $\parallel b\parallel\leq c$, then there is a fixed point at $\sigma=1$. That is, there exists $b_1\in \mathcal{B}$ with $b_1=\mathcal{F}(b_1,1)$.
\end{theorem}

To set up the fixed point method write $u^r=\tilde{v}+w^r$ and $f=\Delta \tilde{v}+\mathcal{K}|\nabla \tilde{v}|$, where $\tilde{v}=v+v_0$ with $v_0\in C^{\infty}(M)$ a fixed function satisfying $v_0 =h$ on $\partial M$ and $v_0\equiv 0$ on $M_{end}$. Then boundary value problem \eqref{ty}, \eqref{ty0} becomes
\begin{equation}\label{ty1}
\Delta w^r=-\mathcal{K}\left(|\nabla u^r|-|\nabla \tilde{v}|\right)-f
=-\mathcal{K}\left(\frac{\nabla(w^r+2\tilde{v})}{|\nabla(w^r+\tilde{v})|+|\nabla \tilde{v}|}\right)\cdot\nabla w^r-f\quad\text{ on }\quad M_r,
\end{equation}
\begin{equation}\label{ty1b}
w^r=0 \quad\text{ on }\quad\partial M_r.
\end{equation}
Let $C^{2,\alpha}_0(M_r)$ denote the space of $C^{2,\alpha}(M_r)$ functions which
vanish on the boundary, and observe that $\Delta^{-1}:C^{2,\alpha}_0(M_r)\rightarrow C^{0,\alpha}(M_r)$ is an isomorphism. Now set
\begin{equation}\label{ty2}
\mathcal{F}(w,\sigma)=\sigma\Delta^{-1}\left[-\mathcal{K}
\left(\frac{\nabla(w+2\tilde{v})}{|\nabla(w+\tilde{v})|+|\nabla \tilde{v}|}\right)\cdot\nabla w-f\right]=:\sigma \Delta^{-1} F(w),\quad\quad w\in C^{1,\alpha}_0(M_r).
\end{equation}
Observe that $F(w)\in C^{0,\alpha}(M_r)$ and hence $\mathcal{F}(w,\sigma)\in C^{2,\alpha}_0(M_r)$. We choose $\mathcal{B}=C^{1,\alpha}_0(M_r)$ and note that
the composition
\begin{equation}
C^{1,\alpha}_0(M_r)\mathop{\longrightarrow}^{F}C^{0,\alpha}(M_r)
\mathop{\longrightarrow}^{\Delta^{-1}} C^{2,\alpha}_0(M_r)
\mathop{\longrightarrow}^{\iota} C^{1,\alpha}_0(M_r),
\end{equation}
yields a compact map $\mathcal{F}:\mathcal{B}\times [0,1]\rightarrow\mathcal{B}$ since the first two pieces $F$ and $\Delta^{-1}$ are bounded while the inclusion $\iota$ is compact. Furthermore, finding a fixed point
$w^r=\mathcal{F}(w^r,1)$ is equivalent to solving \eqref{ty1}, \eqref{ty1b} in $C^{2,\alpha}_0(M_r)$ by elliptic regularity. Then $u^r=\tilde{v}+w^r$ is the desired solution of \eqref{ty}, \eqref{ty0}.

It remains to establish the a priori estimate $|w_{\sigma}|_{C^{1,\alpha}(M_r)}\leq c$, independent of $\sigma$, for a fixed point $w_{\sigma} =\mathcal{F}(w_{\sigma}, \sigma)$. Such a fixed point satisfies \eqref{ty1}, \eqref{ty1b} with $\mathcal{K}$ and $f$ replaced by $\sigma\mathcal{K}$ and $\sigma f$. This may be viewed as a linear equation with coefficients that depend on the solution. However, since the coefficients remain uniformly bounded independent of the solution, $L^p$ estimates for linear elliptic equations may be applied to obtain
\begin{equation}
\parallel w_{\sigma}\parallel_{W^{2,p}(M_r)}\leq C\left(\parallel f\parallel_{L^p(M_r)}+\parallel w_\sigma\parallel_{L^{p}(M_r)}\right),
\end{equation}
where $W^{l,p}$ denotes the Sobolev space with $l$ weak derivatives in $L^p$, $p>1$.
%By the Gagliardo-Nirenberg interpolation inequality
%\begin{equation}
%\parallel \nabla w_\sigma \parallel_{L^{p}(M_r)}\leq C_1 \parallel %w_\sigma\parallel^{1/2}_{W^{2,p}(M_r)}\parallel w_\sigma\parallel^{1/2}_{L^p(M_r)}
%\leq \frac{\epsilon C_1}{2}\parallel w_\sigma\parallel_{W^{2,p}(M_r)}
%+\frac{C_1}{2\epsilon}\parallel w_\sigma\parallel_{L^p(M_r)}.
%\end{equation}
%Thus, choosing $\epsilon>0$ sufficiently small yields
%\begin{equation}
%\parallel w_t\parallel_{W^{2,p}(B_r)}\leq C_2\left(\parallel %f\parallel_{L^p(B_r)}+\parallel w_t\parallel_{L^p(B_r)}\right).
%\end{equation}
Moreover, since the coefficient of the zeroth order term in \eqref{ty1} vanishes, the maximum principle is valid and leads to a $C^0$ bound for $w_\sigma$ which in turn gives a bound for $\parallel w_\sigma\parallel_{L^p(M_r)}$. Hence we obtain the a priori estimate
\begin{equation}\label{p0}
\parallel w_\sigma\parallel_{W^{2,p}(M_r)}\leq C_3,
\end{equation}
independent of $\sigma$.
%By differentiating the equation \eqref{ty1} and using the estimate \eqref{p0} we %obtain an analogous estimate in $W^{3,p}(M_r)$
%$\parallel w_t\parallel_{W^{3,p}(B_r)}\leq C_4$.
According to the Sobolev embedding $W^{2,p}(M_r)\hookrightarrow C^{1,\alpha}(M_r)$ for $p$ sufficiently large, we obtain
\begin{equation}
|w_\sigma|_{C^{1,\alpha}(M_r)}\leq C_5,
\end{equation}
independent of $\sigma$. The Leray-Schauder theorem may now be applied to obtain a fixed point at $\sigma=1$.

\subsection{Barriers}\label{barrier}

Rotationally symmetric asymptotic barrier functions will be constructed to obtain uniform bounds on the solutions $w^r$ of \eqref{ty1}, \eqref{ty1b} independent of $r$. To this end, in the asymptotically flat region set
\begin{equation}
\overline{w}(r)=\lambda r^{-\beta},\quad\quad \overline{w}'(r)=-\lambda\beta r^{-1-\beta},\quad\quad
\overline{w}''(r)=\lambda\beta(1+\beta)r^{-2-\beta},
\end{equation}
for $\beta\in(0,1)$ and where $\lambda>0$ is a constant to be chosen sufficiently large.
Using the level sets of $r$, the metric may be expressed as
\begin{equation}
g=|\partial_r|^2 dr^2 +g_r=|\nabla r|^{-2}dr^2+g_r,
\end{equation}
where $g_r$ is the induced metric on the coordinate spheres $S_r$.
If $\upsilon$ denotes the unit outer normal to the coordinate spheres then
\begin{equation}
\upsilon=|\partial_r|^{-1}\partial_r=|\nabla r|\partial_r,\quad\quad\quad
|\nabla r|^2=g^{ij}\partial_i r\partial_j r=\frac{g^{ij}x^i x^j}{r^2}
=1+O_2(r^{-q}),
\end{equation}
and the Laplacian becomes
\begin{equation}
\Delta\overline{w}=\nabla_{\upsilon}^2 \overline{w}+H_{S_r}\upsilon(\overline{w}),
\end{equation}
where $H_{S_r}$ denotes mean curvature. Observe that
\begin{equation}
\nabla_{\upsilon}^2 \overline{w}=\upsilon^i \upsilon^j \nabla_{ij}\overline{w}
=|\nabla r|^2 \nabla_{rr}\overline{w}
=|\nabla r|^2 \left(\overline{w}''-\Gamma_{rr}^r \overline{w}'\right),
\end{equation}
and
\begin{equation}
\Gamma_{rr}^r
=\frac{1}{2}g^{rr}\partial_{r}g_{rr}=-\partial_{r}\log|\nabla r|,
\end{equation}
so that
%\begin{equation}
%\nabla_{\nu}^2 \overline{w}=|\nabla r|^2 \left(\overline{w}''+\frac{\partial_r %|\nabla r|}{|\nabla r|}\overline{w}'\right).
%\end{equation}
%Therefore
\begin{align}
\begin{split}
\Delta\overline{w}=&|\nabla r|^2 \overline{w}''
+|\nabla r|\left(H_{S_r}+\partial_r|\nabla r|\right)\overline{w}'\\
=&\left(1+O(r^{-q})\right)\left(\overline{w}''
+\frac{2}{r}\overline{w}'+O(r^{-q-1})\overline{w}'\right)\\
=&-\lambda\beta(1-\beta)r^{-2-\beta}\left(1+O(r^{-q})\right).
\end{split}
\end{align}
Furthermore
\begin{equation}
\left|\overline{\mathcal{K}}(w^r)\cdot\nabla\overline{w}\right|
:=\left|\mathcal{K}\left(\frac{\nabla(w^r+2\tilde{v})}
{|\nabla(w^r+\tilde{v})|+|\nabla \tilde{v}|}\right)\cdot\nabla\overline{w}\right|\leq Cr^{-q-1}|\overline{w}'|
= C\lambda\beta r^{-q-2-\beta}.
\end{equation}
It follows that
\begin{equation}
L\overline{w}:=\Delta\overline{w}
+\overline{\mathcal{K}}(w^r)\cdot\nabla\overline{w}
=-\lambda\beta(1-\beta)r^{-2-\beta}\left(1+O(r^{-q})\right).
\end{equation}

Consider now the asymptotics for $f$. According to \eqref{vequation}, \eqref{vequationb} we have
\begin{equation}
|f|=\left|\Delta v+\mathcal{K}|\nabla v|\right|=|\mathcal{K}||1-|\nabla v||\leq
C_1 r^{-2q-1}=C_1 r^{-2-\beta},
\end{equation}
by setting $\beta=2q-1>0$. Therefore, given a large radius $r_0$, it holds that
\begin{equation}
L\overline{w}\leq -f \quad\quad\text{ for }\quad\quad r>r_0
\end{equation}
if $\lambda$ is sufficiently large. Hence, $\overline{w}$ is a super-solution
of \eqref{ty1} on $M_r \setminus M_{r_0}$.

In order to obtain a global barrier let $\tilde{w}^{r_0}$ solve \eqref{ty1}, \eqref{ty1b} on $M_{r_0}$ with $\overline{\mathcal{K}}(\tilde{w}^{r_0})$ replaced by $\overline{\mathcal{K}}(w^{r})$, noting that this is a linear boundary value problem. Next define
\begin{equation}
\hat{w}_{\lambda}=
\begin{cases}
\tilde{w}:=\tilde{w}^{r_0}+\lambda r_0^{-\beta} & \text{on}\quad M_{r_0}, \\
\overline{w} & \text{on}\quad M_r \setminus M_{r_0}.
\end{cases}
\end{equation}
This function is smooth everywhere, except at $S_{r_0}$ where it is continuous, and is a super-solution for \eqref{ty1} on $M_{r_0}$ and $M_r \setminus M_{r_0}$ separately. Furthermore we have
\begin{equation}
\partial_r \tilde{w}> \partial_r \overline{w}\quad\quad\text{ at }\quad\quad S_{r_0},
\end{equation}
if $\lambda$ is sufficiently large (independent of $r$), and this allows for an application of the weak maximum principle. To see this, let $\phi\in C^{\infty}_{c}(M_r)$ be a nonnegative test function and observe that
\begin{align}
\begin{split}
0=&-\int_{M_{r_0}}\phi L(\tilde{w}-w^r)dV\\
=&\int_{M_{r_0}}\left(\nabla\phi\cdot\nabla(\tilde{w}-w^r)
-\phi\overline{\mathcal{K}}(w^r)\cdot\nabla(\tilde{w}-w^r)\right)dV
-\int_{S_{r_0}}\phi\partial_{r}(\tilde{w}-w^r)dA,
\end{split}
\end{align}
and
\begin{align}
\begin{split}
0\leq&-\int_{M_{r}\setminus M_{r_0}}\phi L(\overline{w}-w^r)dV\\
=&\int_{M_{r}\setminus M_{r_0}}\left(\nabla\phi\cdot\nabla(\overline{w}-w^r)
-\phi\overline{\mathcal{K}}(w^r)\cdot\nabla(\overline{w}-w^r)\right)dV
+\int_{S_{r_0}}\phi\partial_{r}(\overline{w}-w^r)dA,
\end{split}
\end{align}
so that upon adding these two inequalities
\begin{equation}
\int_{M_r}\left(\nabla\phi\cdot\nabla(\hat{w}_{\lambda}-w^r)
-\phi\overline{\mathcal{K}}(w^r)\cdot\nabla(\hat{w}_{\lambda}-w^r)\right)dV
\geq\int_{S_{r_0}}\phi(\partial_{r}\tilde{w}
-\partial_{r}\overline{w})dA\geq 0.
\end{equation}
Hence, according to \cite[Theorem 8.1]{GilbargTrudinger} the weak maximum principle yields
\begin{equation}
\inf_{M_r}(\hat{w}_{\lambda}-w^r)\geq
\inf_{\partial M_r}(\hat{w}_{\lambda}-w^r)\geq 0.
\end{equation}
A similar argument with $\hat{w}_{-\lambda}$ yields a lower bound, and therefore
\begin{equation}
\hat{w}_{-\lambda}< w^r <\hat{w}_{\lambda}\quad\text{ on }\quad M_{r}.
\end{equation}
Consequently we obtain a global $C^0$ estimate for $w^r$ independent of $r$.

\subsection{The global existence result}

Here we will show that $w^r$ subconverges on compact subsets as $r\rightarrow\infty$ to a $C^{2,\alpha}$ solution of \eqref{ty1} on all of $M$. In the previous subsection a uniform $C^0$ estimate was achieved. Consider
\eqref{ty1} as a linear equation with coefficients depending on $w^r$ but which are uniformly bounded, and apply the local $L^p$ estimates to find
\begin{equation}
\parallel w^r \parallel_{W^{2,p}(\Omega)}\leq C\left(\parallel f\parallel_{L^{p}(\Omega')}+\parallel w^r\parallel_{L^p(\Omega')}\right),
\end{equation}
where $\Omega\subset\subset\Omega'$ are any fixed compact subsets of $M_r$ and $C$ is independent of $r$. The uniform $C^0$ bound implies a uniform $L^p$ bound in $\Omega'$, and therefore
\begin{equation}
\parallel w^r \parallel_{W^{2,p}(\Omega)}\leq C'.
\end{equation}
By Sobolev embedding this yields a uniform $C^{1,\alpha}(\Omega)$ bound, so that in particular the right-hand side of \eqref{ty1} is controlled in $C^{0,\alpha}(\Omega)$. Now applying the local Schauder estimates we obtain the desired $C^{2,\alpha}$ estimate on a subset of $\Omega$, for any $\alpha\in(0,1)$. It follows then by a diagonal argument that there is a subsequence $w^{r_i}$ converging in $C^{2,\alpha}$ on any compact subset of $M$ to a smooth function $w$ which solves
\begin{equation}\label{ty10}
\Delta w=-\mathcal{K}\left(\frac{\nabla(w+2\tilde{v})}{|\nabla(w+\tilde{v})|+|\nabla \tilde{v}|}\right)\cdot\nabla w-f\quad\text{ on }\quad M,
\end{equation}
\begin{equation}
w=0\quad\text{ on }\quad \partial M,\quad\quad\quad
-\overline{w}<w<\overline{w} \quad\text{ on }\quad M\setminus M_{r_0}.
\end{equation}
Finally, by setting $u=\tilde{v}+w$ we obtain the desired solution of \eqref{88}, \eqref{881}.

As mentioned at the start of the section, this global existence result extends in a
straightforward manner to the case of multiple asymptotically flat ends $M_{end}^\ell$, $\ell=1,\ldots,\ell_0$. For this situation
let $a_i^{\ell} x^i$ be a linear function of the asymptotic coordinates in the end $M_{end}^{\ell}$, with $\sum_i (a_i^\ell)^2=1$, and let $h\in C^{\infty}(\partial M)$. Then the background function satisfies
\begin{equation}\label{vequation1}
\Delta v=-\mathcal{K}\quad\quad\text{ on }\quad\quad M,
\end{equation}
\begin{equation}\label{vequationb1}
v=0\quad\text{ on }\quad \partial M,\quad\quad v=a_i^{\ell} x^i +O_2(r^{1-q})\quad\text{ as }\quad r\rightarrow\infty\quad\text{ in }\quad M_{end}^{\ell},\text{ }\text{ }\ell=1,\ldots,\ell_0.
\end{equation}

\begin{theorem}\label{Thm:PDE}
Suppose that $(M,g,k)$ is a smooth asymptotically flat initial data set with (possibly empty) boundary $\partial M$, and $h\in C^{\infty}(\partial M)$.
Let $v$ be a solution of \eqref{vequation1} and \eqref{vequationb1}. Then for each $\alpha\in(0,1)$ there exists a solution $u\in C^{2,\alpha}(M)$ of the spacetime harmonic function equation
\begin{equation}\label{1231}
\Delta u+\mathcal{K}|\nabla u|=0\quad\quad\text{ on }\quad\quad M,
\end{equation}
such that
\begin{equation}\label{12341}
u=h\quad\text{ on }\quad \partial M,\quad\quad\quad\quad u=v +O_2(r^{1-2q})\quad\text{ as }\quad r\rightarrow\infty.
\end{equation}
The solution $u$ is unique among those which satisfy \eqref{12341}.
\end{theorem}

\begin{proof}
The existence portion was proven in the discussion above, while the uniqueness follows from the maximum principle in the same manner as the barrier argument at the end of Section \ref{barrier}. Lastly, the decay of derivatives in the asymptotic ends may be established analogously to \cite[Proposition 3]{SchoenYauII}.
\end{proof}

\section{Controlling the Level Set Topology}
\label{sec4} \setcounter{equation}{0}
\setcounter{section}{5}

In order to apply the integral inequality Proposition \ref{SpacetimeSternCompact} successfully, it is important to ensure that the Euler characteristic of regular level sets for the spacetime harmonic function does not exceed 1. In this section, we show that it is possible to choose the spacetime harmonic function, by carefully selecting its Dirichlet data, to achieve this goal for the level sets. Since this will be employed for generalized exterior regions, here we consider asymptotically flat initial data $(M,g,k)$ with a single asymptotic end, although the boundary may have several components $\partial M=\cup_{i=1}^{n}\partial_i M$. Let $v$ solve \eqref{vequation1}, \eqref{vequationb1} and consider the Dirichlet problem
\begin{equation}\label{sharmonictop}
\Delta u_{\mathbf{c}}+\mathcal{K}|\nabla u_{\mathbf{c}}|=0\quad\quad\text{ on }\quad\quad M,
\end{equation}
\begin{equation}\label{sharmonictopb}
u_{\mathbf{c}}=c_i\quad\text{ on }\quad \partial_i M,\text{ }\text{ }i=1,\ldots,n,\quad\quad\quad\quad u_{\mathbf{c}}=v +O_2(r^{1-2q})\quad\text{ as }\quad r\rightarrow\infty,
\end{equation}
where $\mathbf{c}=(c_1,\ldots,c_n)$ are constants. The following is a technical preliminary result that indicates how to choose the constants $\mathbf{c}$ in order to achieve the main topological conclusions of Theorem \ref{shfexistence} concerning level sets, as well as to aid with the computation of boundary terms in the integral inequality Proposition \ref{SpacetimeSternCompact}.

\begin{lemma}\label{l:dirichlet}
Let $u_{\mathbf{c}}$ be the solution of \eqref{sharmonictop}, \eqref{sharmonictopb} given by Theorem \ref{Thm:PDE}.

\begin{enumerate}
\item Let $(-1)^{\varsigma_i}$, $\varsigma_i\in\{0,1\}$ be a choice of sign associated with each boundary component $i=1,\ldots,n$. There exists a set of constants $\mathbf{c}$ such that for each boundary component there is a point $y_i\in\partial_i M$ with $|\nabla u_{\mathbf{c}}(y_i)|=0$, and in addition $(-1)^{\varsigma_i}\partial_{\upsilon}u_{\mathbf{c}}\geq 0$ on $\partial_i M$, where $\upsilon$ is the unit normal to $\partial M$ pointing outside $M$.

\item If $v\neq 0$, then within each boundary component $\partial_i M$, the set of points at which $|\nabla u_{\mathbf{c}}|=0$ is nowhere dense.
\end{enumerate}
\end{lemma}

\begin{proof}
A priori estimates established in the previous section show that $u_{\mathbf{c}}$ is continuously differentiable in $\mathbf{c}$. Set $u_i:=\partial_{c_i}u_{\mathbf{c}}$ and observe that these functions satisfy
\begin{equation}\label{l,m}
\Delta u_i
+\mathcal{K}\frac{\nabla u_{\mathbf{c}}}{|\nabla u_{\mathbf{c}}|}\cdot\nabla u_i=0\quad\quad\text{ on }\quad\quad M,
\end{equation}
\begin{equation}\label{l,m1}
u_i=\delta_{ij}\quad\text{ on }\quad \partial_j M,\text{ }\text{ }j=1,\ldots,n,\quad\quad\quad\quad u_i=O(r^{1-2q})\quad\text{ as }\quad r\rightarrow\infty.
\end{equation}
Clearly the set of functions $\{u_1,\ldots, u_n\}$ is linearly independent. Pick $y_i\in\partial_i M$, $i=1,\ldots,n$ and set $\mathbf{y}=(y_1,\ldots,y_n)$. We claim that the Jacobian matrix
\begin{equation}
U(\mathbf{c},\mathbf{y})=
\begin{bmatrix}
    \partial_{\upsilon}u_1(y_1) & \partial_{\upsilon}u_2(y_1) & \dots  & \partial_{\upsilon}u_n(y_1) \\
    \partial_{\upsilon}u_1(y_2) & \partial_{\upsilon}u_2(y_2) & \dots  & \partial_{\upsilon}u_n(y_2) \\
    \vdots & \vdots & \ddots & \vdots \\
    \partial_{\upsilon}u_1(y_n) & \partial_{\upsilon}u_2(y_n) & \dots  & \partial_{\upsilon}u_n(y_n)
\end{bmatrix}
\end{equation}
is invertible, where $\upsilon$ is the unit outer normal to $\partial M$. Suppose by way of contradiction that it is not invertible. Then there exist constants $b_i$, $i=1,\ldots,n$, not all zero, such that $u=\sum_{i=1}^n b_i u_i$ satisfies $\partial_{\upsilon}u(y_j)=0$, $j=1,\ldots,n$. Note that the function $u$ satisfies
\begin{equation}
\Delta u
+\mathcal{K}\frac{\nabla u_{\mathbf{c}}}{|\nabla u_{\mathbf{c}}|}\cdot\nabla u=0\quad\text{ on }\quad M,
\end{equation}
\begin{equation}
u=b_j\quad\text{ on }\quad \partial_j M,\text{ }\text{ }j=1,\ldots,n,\quad\quad\quad \quad u=O(r^{1-2q})\quad\text{ as }\quad r\rightarrow\infty.
\end{equation}
Since $b_i$ are not all zero, we have that $u\neq 0$. On the other hand, by the maximum principle either the global max or min must be achieved on $\partial_{i_0} M$ for some $i_0$. By the Hopf lemma, we then have $\partial_{\upsilon}u(y_{i_0})\neq 0$. However this contradicts the basic property of $u$ described above. It follows that $U$ is invertible.

We now show that $U(\mathbf{c},\mathbf{y})$ stays uniformly bounded and away from being singular. To see this, suppose that for a sequence $\{(\mathbf{c}_l,\mathbf{y}_l)\}_{l=1}^{\infty}$ either $\parallel U(\mathbf{c}_l,\mathbf{y}_l)\parallel\rightarrow \infty$ or $\det U(\mathbf{c}_l,\mathbf{y}_l)\rightarrow 0$. Observe that the solutions $\partial_{c_i}u_{\mathbf{c}_l}$ of \eqref{l,m}, \eqref{l,m1} are uniformly controlled in $W^{2,p}_{loc}(M)$ by the $L^p$ estimates, since the first order coefficients remain uniformly bounded. It follows that there is subsequential convergence in $C^{1,\alpha}(M)$ to a solution $\partial_{c_i}u_{\infty}$, which by elliptic regularity lies in $C^{2,\alpha}(M)$. Consequently, using that the sequence $\{\mathbf{y}_l\}\subset \Pi_{i=1}^n \partial_i M$ lies in a compact set, we find that there is subconvergence $U(\mathbf{c}_l,\mathbf{y}_l)\rightarrow U(\infty)$. However, the arguments of the previous paragraph show that $U(\infty)$ is invertible, and this contradiction yields the desired conclusion. In particular, $U^{-1}(\mathbf{c},\mathbf{y})$ is uniformly bounded.

Consider the map $\mathfrak{U}:\mathbb{R}^n \rightarrow\mathbb{R}^n$ given by
\begin{equation}
\mathfrak{U}(c_1,\ldots,c_n)=(\partial_{\upsilon}u_{\mathbf{c}}(y_1(\mathbf{c})),\ldots,
\partial_{\upsilon}u_{\mathbf{c}}(y_n(\mathbf{c}))),
\end{equation}
where $y_i(\mathbf{c})\in\partial_i M$ is a point at which $\partial_{\upsilon}u_{\mathbf{c}}$ achieves its: minimum over this component when $\varsigma_i=0$, or maximum over this component when $\varsigma_i=1$. Observe that $\mathfrak{U}$ is continuous. Moreover, it will be shown that this function is differentiable in certain directions, and the matrix $U$ will play a role similar to a Jacobian for $\mathfrak{U}$. Set
$\mathbf{p}_0=(\partial_{\upsilon}u_{0}(y_1(0)),\ldots,
\partial_{\upsilon}u_{0}(y_n(0)))$, and let $\mathbf{p}(t)\subset\mathbb{R}^n$ be a smooth curve emanating from $\mathbf{p}_0 =\mathbf{p}(0)$ and ending at $\mathbf{p}(1)=0$. We claim that there is a smooth curve $\mathbf{c}(t)$, $t\in[0,1]$, emanating from $\mathbf{c}(0)=0$, such that $\mathfrak{U}(\mathbf{c}(t))=\mathbf{p}(t)$. To find this solve the ODE initial value problem
\begin{equation}\label{nfr}
\mathbf{c}'(t)=U^{-1}(\mathbf{c}(t),\mathbf{y}(\mathbf{c}(t)))\mathbf{p}'(t),\quad\quad\quad\mathbf{c}(0)=0.
\end{equation}
Observe that global existence holds since $U^{-1}(\mathbf{c},\mathbf{y}(\mathbf{c}))$ is uniformly bounded independent of $\mathbf{c}$.

We will now show that $\mathfrak{U}(\mathbf{c}(t))$ is differentiable. Let $y_i(\mathbf{c})$ be a minimum point for $\partial_{\upsilon}u_{\mathbf{c}}$ on $\partial_i M$, and $0\leq s<t\leq 1$; a similar argument holds or a maximum point. Then
\begin{align}
\begin{split}
\partial_{\upsilon}u_{\mathbf{c}(t)}(y_i(\mathbf{c}(t)))
-\partial_{\upsilon}u_{\mathbf{c}(s)}(y_i(\mathbf{c}(s)))
=&\left[\partial_{\upsilon}u_{\mathbf{c}(t)}(y_i(\mathbf{c}(t)))
-\partial_{\upsilon}u_{\mathbf{c}(s)}(y_i(\mathbf{c}(t)))\right]\\
&+\left[\partial_{\upsilon}u_{\mathbf{c}(s)}(y_i(\mathbf{c}(t)))
-\partial_{\upsilon}u_{\mathbf{c}(s)}(y_i(\mathbf{c}(s)))\right]\\
\geq&\partial_{\upsilon}u_{\mathbf{c}(t)}(y_i(\mathbf{c}(t)))
-\partial_{\upsilon}u_{\mathbf{c}(s)}(y_i(\mathbf{c}(t)))\\
=&\sum_j\partial_{\upsilon}\partial_{c_j}u_{\mathbf{c}(t)}(y_{i}(\mathbf{c}(t))
c'_j(t)(t-s)+o(t-s)\\
=&p'_i(t)(t-s) +o(t-s),
\end{split}
\end{align}
and
\begin{align}
\begin{split}
\partial_{\upsilon}u_{\mathbf{c}(t)}(y_i(\mathbf{c}(t)))
-\partial_{\upsilon}u_{\mathbf{c}(s)}(y_i(\mathbf{c}(s)))
=&\left[\partial_{\upsilon}u_{\mathbf{c}(t)}(y_i(\mathbf{c}(t)))
-\partial_{\upsilon}u_{\mathbf{c}(t)}(y_i(\mathbf{c}(s)))\right]\\
&+\left[\partial_{\upsilon}u_{\mathbf{c}(t)}(y_i(\mathbf{c}(s)))
-\partial_{\upsilon}u_{\mathbf{c}(s)}(y_i(\mathbf{c}(s)))\right]\\
\leq&\partial_{\upsilon}u_{\mathbf{c}(t)}(y_i(\mathbf{c}(s)))
-\partial_{\upsilon}u_{\mathbf{c}(s)}(y_i(\mathbf{c}(s)))\\
=&\sum_j\partial_{\upsilon}\partial_{c_j}u_{\mathbf{c}(s)}(y_{i}(\mathbf{c}(s))
c'_j(s)(t-s)+o(t-s)\\
=&p'_i(s)(t-s)+o(t-s),
\end{split}
\end{align}
where we have used Taylor's theorem and \eqref{nfr} with the notation $\mathbf{p}(t)=(p_1(t),\ldots,p_n(t))$. Dividing both sides of these equations by $t-s$ and letting $t\rightarrow s$ shows that $\mathfrak{U}(\mathbf{c}(t))$ is differentiable, and
\begin{equation}
\frac{d}{dt}\mathfrak{U}(\mathbf{c}(t))=\mathbf{p}'(t).
\end{equation}
Integrating this equation then gives the desired relation. We now have $\mathfrak{U}(\mathbf{c}(1))=0$, so that $\mathbf{c}(1)$ is the claimed set of constants such that $\partial_{\upsilon}u_{\mathbf{c}(1)}(y_i(\mathbf{c}(1)))=0$, $i=1,\ldots,n$. This completes the proof of (1).

Consider now part (2). Suppose that the set within $\partial_i M$ on which $|\nabla u_{\mathbf{c}}|=0$ has a nonempty interior. Then since equation \eqref{sharmonictop} may be viewed as a linear equations with bounded coefficients, the unique continuation result \cite[Theorem 1.7]{TaoZhang} applies to show that $u_{\mathbf{c}}\equiv const$. This contradicts the assumption that $v\neq 0$. Since the set on which $|\nabla u_{\mathbf{c}}|=0$ is also closed, it follows that it is nowhere dense.
\end{proof}

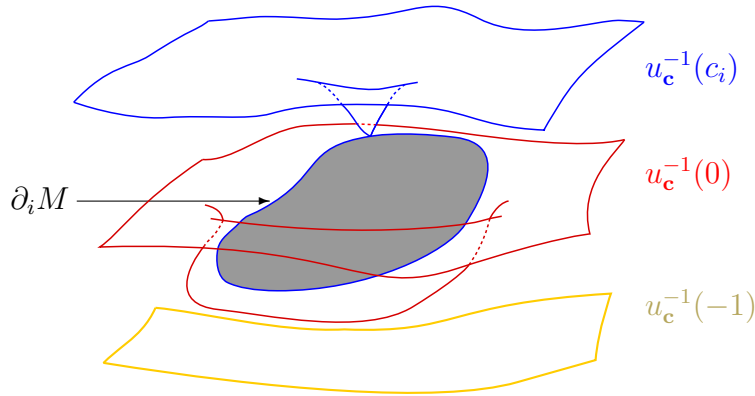
\begin{figure}[h!]
\begin{picture}(0,0)
\put(-20,70){\large{$\partial_i M$}}
\put(5,73){\vector(1,0){73}}
\put(220,120){{\color{blue}\large{$u_{\mathbf{c}}^{-1}(c_i)$}}}
\put(220,80){{\color{red}\large{$u_{\mathbf{c}}^{-1}(0)$}}}
\put(220,30){{\color{brass}\large{$u_{\mathbf{c}}^{-1}(-1)$}}}
\end{picture}
\includegraphics[scale=.4]{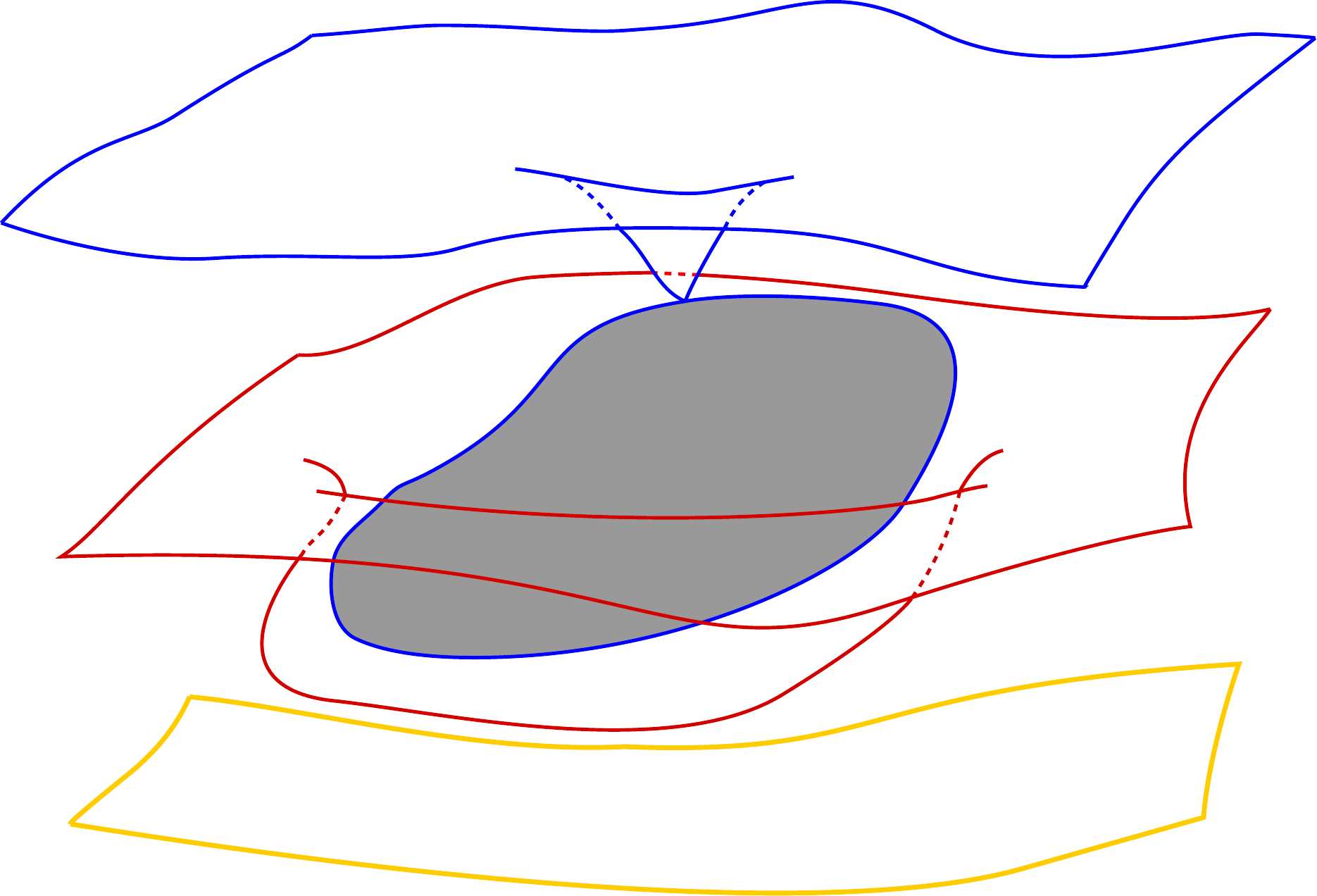}
\caption{Possible level sets of the function $u_{\mathbf{c}}$ constructed in Lemma \ref{l:dirichlet}.}\label{pic:section5}
\end{figure}

We are now in a position to establish the main topological result concerning regular level sets of the spacetime harmonic function $u_{\mathbf{c}}$ arising from Lemma \ref{l:dirichlet}. This will later be employed in generalized exterior regions which have a simplified topology, although we do not use here the MOTS and MITS condition on the boundary of such regions.

\begin{theorem}\label{shfexistence}
Let $(M,g,k)$ be a smooth asymptotically flat initial data set having a single asymptotic end, and satisfying $H_2(M,\partial M;\mathbb{Z})=0$. Let $u_{\mathbf{c}}$ be the solution of \eqref{sharmonictop}, \eqref{sharmonictopb} with $\mathbf{c}$ given by Lemma \ref{l:dirichlet}.
Then all regular level sets of $u_{\mathbf{c}}$ are connected and noncompact with a single end modeled on $\mathbb{R}^2\setminus B_1$.
In particular, if $\Sigma_t$ is a regular level set then its Euler characteristic
satisfies $\chi(\Sigma_t)\leq 1$.
\end{theorem}

\begin{proof}
Let $\Sigma_t=u^{-1}_{\mathbf{c}}(t)$ be a regular level set, and suppose that there is a compact connected component $\Sigma'_{t}\subset\Sigma_t$. Note that $\Sigma'_{t}$ is a 2-sided properly embedded submanifold. Since $H_2(M,\partial M;\mathbb{Z})=0$ the boundary cycles $\partial_i M$, $i=1,\ldots,n$ generate $H_2(M)$. Thus, either $\Sigma'_t$ is homologous to zero or it is homologous to a sum of boundary cycles. In the former case $\Sigma'_{t}$ bounds a compact region of $M$, and since the spacetime harmonic function equation admits a maximum principle the solution $u_{\mathbf{c}}\equiv t$ in this region. This, however, contradicts the assumption that $t$ is a regular value. So now consider the later case in which $[\Sigma'_t]$ can be represented as the sum of boundary classes $\sum_{i\in I}[\partial_i M]$, for some index set $I$. Let $D\subset M$ denote the compact region bounded by $\Sigma'_t \cup\left(\cup_{i\in I}\partial_i M\right)=\partial D$. Since the maximum and minimum of $u_{\mathbf{c}}$ on $D$ are achieved on the boundary, it follows that either the max or min is achieved on $\partial_{i_0} M$, for some $i_0\in I$. In particular, this max or min is achieved at $y_{i_0}\in \partial_{i_0} M$. Next observe that the Hopf lemma applies to the spacetime harmonic function equation, since the nonlinear first order part may be expressed as a linear term with bounded coefficients, and therefore $\partial_{\upsilon}u_{\mathbf{c}} (y_{i_0})\neq 0$. However this contradicts the fact that $y_{i_0}$ is a critical point for $u_{\mathbf{c}}$, as stated in Lemma \ref{l:dirichlet}.
We conclude that all components of $\Sigma_t$ are noncompact. Moreover $\Sigma_t$ is a closed subset of $M$, since it is properly embedded.
Therefore if any component of $\Sigma_t$ stays within $M_r$ (see Section \ref{mr}), %the compact region bounded by the coordinate sphere $S_r\subset M_{end}$,
it must be compact which is a contradiction. It follows that each component must extend beyond $S_r$ for all $r$.

The asymptotics of $u_{\mathbf{c}}\sim a_i x^i$ in the end $M_{end}$ imply that for all sufficiently large $r$ the level set $\Sigma_t$ stays within a slab $\{x\in M\setminus M_r \mid t-C<a_i x^i<t+C\}$, for some constant $C$. Indeed, by the implicit function theorem $\Sigma_t$ may be presented uniquely in this region as a graph over the plane $t=a_i x^i$. Hence, $\Sigma_t$ is connected and has a single end modeled on $\mathbb{R}^2\setminus B_1$.
\end{proof}

\section{Proof of the Mass Lower Bound}
\label{sec5} \setcounter{equation}{0}
\setcounter{section}{6}

Let $(M,g,k)$ be a complete asymptotically flat initial data set for the
Einstein equations, having generalized exterior region $M_{ext}$ associated with a particular end $M_{end}$ and given by Proposition \ref{exterior}; for convenience we will continue denoting the metric and extrinsic curvature on $M_{ext}$ by $(g,k)$. Suppose that $x=(x^1,x^2,x^3)$ are \textit{spacetime harmonic coordinates} on $M_{ext}$. This means that each function $x^l$ satisfies \eqref{sharmonictop}, \eqref{sharmonictopb} and is given by Theorem \ref{Thm:PDE} and Lemma \ref{l:dirichlet} (1), with asymptotics $x^l\sim \tilde{x}^l$ for some given asymptotically flat coordinate system $\tilde{x}=(\tilde{x}^1,\tilde{x}^2,\tilde{x}^3)$ on $M_{end}$. More precisely, by Lemma \ref{l:dirichlet}(1) we may choose the sign of the normal derivative at each boundary component $\partial_i M_{ext}$, $i=1,\ldots,n$ so that:
\begin{equation}\label{boundarybehavior}
\partial_{\upsilon} x^l\leq 0 \text{ on }\partial_i M_{ext}\text{ if }\theta_+\left(\partial_i M_{ext}\right)=0,\quad
\partial_{\upsilon}x^l\geq 0 \text{ on }\partial_i M_{ext}\text{ if }\theta_-\left(\partial_i M_{ext}\right)=0,\text{ } l=1,2,3.
\end{equation}
Note that although $x^l$ are referred to as spacetime harmonic coordinates and are defined on all of $M_{ext}$, they are only guaranteed to form a coordinate system in $M_{end}$. Observe that due to the asymptotic expansion in Theorem \ref{Thm:PDE}, the ADM energy and linear momentum computed in spacetime harmonic coordinates will agree with the computation in any other valid asymptoticaly flat coordinate system \cite{Bartnik}.

For $L>0$ sufficiently large define the cylindrical boundaries
\begin{align}
\begin{split}
&D_L^{\pm}=\{x\in M_{end}\mid (x^2)^2+(x^3)^2\leq L^2,\; x^1=\pm L\},\\
&T_L=\{x\in M_{end}\mid (x^2)^2+(x^3)^2=L^2,\; |x^1|\leq L\},\\
&C_L=D^+_L\cup T_L\cup D^-_L,
\end{split}
\end{align}
and denote by $\Omega_L$ the bounded component of $M_{ext}\setminus C_L$, so that
$\partial\Omega_L=C_L\sqcup\partial M_{ext}$. Let $u=x^1$ be the spacetime harmonic function described above, and set $\Sigma_t =u^{-1}(t)$ as well as $\Sigma_t^L=\Sigma_t\cap\Omega_L$. If $t$ is a regular value of $u$, then $\partial\Sigma_t^L$ lies entirely within $C_L$, due to the fact that $u$ has critical points on each component $\partial_i M_{ext}$, $i=1,\ldots,n$. Note also that the regular level sets $\Sigma_t^L$ meet $T_L$ transversely, and by Theorem \ref{shfexistence}, $\Sigma_t^L$ has only one component so that $\chi(\Sigma_t^L)\leq 1$. Therefore we may apply Proposition \ref{SpacetimeSternCompact} together with the Gauss-Bonnet theorem to obtain
\begin{align}\label{INEQ1}
\begin{split}
&\frac{1}{2}\int_{\Omega_L}
\left(\frac{|\bar{\nabla}^2 u|^2}{|\nabla u|}+2(\mu-|J|_g)|\nabla u|\right)dV\\
\leq&\int_{-L}^{L}\left(2\pi\chi(\Sigma_t)
-\int_{\Sigma_t^L\cap T_L}\kappa_{t,L}\right)dt
+
\int_{\partial_{\neq 0} \Omega_L}\left(\partial_{\upsilon}|\nabla u|+k(\nabla u,\upsilon)\right)dA
\\
\leq&
4\pi L
-\int_{-L}^{L}\left(\int_{\Sigma_t^L\cap T_L}\kappa_{t,L}\right)dt
+\int_{C_L}\left(\partial_{\upsilon}|\nabla u|+k(\nabla u,\upsilon)\right)dA\\
&
+\int_{\partial_{\neq 0}M_{ext}}\left(\partial_{\upsilon}|\nabla u|+k(\nabla u,\upsilon)\right)dA,
\end{split}
\end{align}
where $\kappa_{t,L}$ is the geodesic curvature of $\Sigma_t^L\cap T_L$ interpreted as the boundary curve in $\Sigma_t$, $\partial_{\neq0}M_{ext}$ denotes the subset of $\partial M_{ext}$ where $|\nabla u|\neq0$, and we have used that $|\nabla u|>0$ on $C_L$.

In what follows we will compute first the outer boundary integral along $C_L$ in the asymptotic end, from which the ADM energy and linear momentum will arise. The inner boundary integral along $\partial_{\neq 0}M_{ext}$ will then be computed and shown to vanish, due to the fact that $\partial M_{ext}$ consists of MOTS and MITS.
Below, the notation $\int_{D_L^{\pm}}\pm f$ will be used to represent $\int_{D_L^{+}}f-\int_{D_L^{-}}f$.

\subsection{Computation of the outer boundary integral}

In \cite[Lemma 6.1]{BrayKazarasKhuriStern}, a computation was carried out in
harmonic coordinates. Each step of the proof applies here without change using spacetime harmonic coordinates, except for equation \cite[(6.9)]{BrayKazarasKhuriStern} where harmonicity was used. By replacing the harmonic function equation with the spacetime harmonic function equation, in this calculation, we find that
\begin{align}
\begin{split}
\int_{C_L}\partial_{\upsilon}|\nabla u| dA=&\int_{D_L^{\pm}}\pm\left(\frac12 \sum_{j}(g_{1j,j}-g_{jj,1})-\mathcal{K}\right)dA\\
&+\frac{1}{2L}\int_{T_L}\left[x^2(g_{21,1}-g_{11,2})+x^3(g_{31,1}-g_{11,3})\right]dA
+O(L^{1-2q}).
\end{split}
\end{align}
Similarly, \cite[Lemma 6.2]{BrayKazarasKhuriStern} may be carried over without
change to the current setting to yield
\begin{align}
\begin{split}
\int_{-L}^L\left(\int_{\Sigma_t\cap T_L}\kappa_{t,L}\right)dt=&4\pi L+\frac{1}{2L}\int_{T_L}\left[x^2(g_{33,2}-g_{23,3})+x^3(g_{22,3}-g_{32,2})\right]dA\\
&+O(L^{1-2q}+L^{-q}).
\end{split}
\end{align}
Next, observe that the outward normal $\upsilon$ to $C_L$ takes the form
\begin{equation}
\upsilon=\begin{cases}
\pm\partial_1+O(|x|^{-q})&\text{ on }D^\pm_L ,\\
\frac{x^2}{L}\partial_2+\frac{x^3}{L}\partial_3+O(|x|^{-q})&\text{ on }T_L .
\end{cases}
\end{equation}
Furthermore
\begin{equation}
\nabla u=g^{i1}\partial_i=\partial_1 +O(|x|^{-q}).
\end{equation}
It follows that
\begin{equation}
k(\nabla u,\upsilon)=\pm k_{11}+O(|x|^{-1-2q})\quad\text{ on }\quad D_{L}^{\pm},
\end{equation}
and
\begin{equation}
k(\nabla u,\upsilon)=\frac{x^2}{L}k_{12}+\frac{x^3}{L}k_{13}+O(|x|^{-1-2q}).
\end{equation}
Finally, combining these computations produces
\begin{align}\label{outerboundary}
\begin{split}
&4\pi L
-\int_{-L}^{L}\left(\int_{\Sigma_t^L\cap T_L}\kappa_{t,L}\right)dt
+\int_{C_L}\left(\partial_{\upsilon}|\nabla u|+k(\nabla u,\upsilon)\right)dA\\
=&\frac{1}{2}\int_{D_L^{\pm}}\pm \sum_j (g_{1j,j}-g_{jj,1})dA\\
&+\frac{1}{2}\int_{T_L}\left[\frac{x^2}{L}(g_{21,1}-g_{11,2})
+\frac{x^3}{L}(g_{31,1}-g_{11,3})\right]dA\\
&+\frac{1}{2}\int_{T_L}\left[\frac{x^2}{L}(g_{23,3}-g_{33,2})
+\frac{x^3}{L}(g_{32,2}-g_{22,3})\right]dA\\
&+\int_{D^\pm_L}\pm \left(k_{11}-\mathcal{K}\right)dA+\int_{T_L}\left(\frac{x^2}{L}k_{12}+\frac{x^3}{L}k_{13}\right)dA
+O(L^{1-2q}+L^{-q})\\
=&\frac{1}{2}\int_{C_L}\sum_j (g_{ij,j}-g_{jj,i})\upsilon^i dA
+\int_{C_L} \left(k_{1i}-(\mathrm{Tr}_g k)g_{1i}\right)\upsilon^i dA
+O(L^{1-2q}+L^{-q}).
\end{split}
\end{align}

\subsection{Computation of the inner boundary integral}

Here we show that the inner boundary integral over $\partial_{\neq 0}M_{ext}$ vanishes, due to boundary behavior of the spacetime harmonic function combined with the fact that each boundary component is either a MOTS or MITS. Moreover, if the boundary components consist of weakly trapped surfaces then the inner boundary integral is nonpositive, which is an advantageous sign with respect to positivity of the ADM energy. Let $\upsilon$ denote the unit normal to a boundary component $\partial_i M_{ext}$, which points outside of $M_{ext}$. Then because $u$ is constant on $\partial_i M_{ext}$, the spacetime harmonic function equation and gradient may be rewritten on this surface as
\begin{equation}
\nabla_{\upsilon}^2 u=H\upsilon(u) - \mathcal{K}|\nabla u|,\quad\quad\quad
\nabla u=\upsilon(u)\upsilon.
\end{equation}
Note that here, the mean curvature $H$ is computed with respect to $-\upsilon$.
Observe that
\begin{equation}
|\nabla u|\partial_{\upsilon}|\nabla u|
=\frac{1}{2}\partial_{\upsilon}|\nabla u|^2
=\frac{1}{2}\partial_{\upsilon}\left(g^{ij}u_i u_j\right)
=u^j \nabla_{j\upsilon} u=\upsilon(u)\nabla_{\upsilon}^2 u,
\end{equation}
and hence
\begin{equation}
\partial_{\upsilon}|\nabla u|=\frac{\upsilon(u)}{|\nabla u|}\nabla_{\upsilon}^2 u
=\frac{\upsilon(u)}{|\nabla u|}\left(H\upsilon(u)-\mathcal{K}|\nabla u|\right)
=H|\upsilon(u)|-\mathcal{K}\upsilon(u).
\end{equation}
Furthermore since
\begin{equation}
\mathcal{K}=k(\upsilon,\upsilon)+\mathrm{Tr}_{\partial M_{ext}}k,\quad\quad\quad\quad k(\nabla u,\upsilon)=k(\upsilon,\upsilon)\upsilon(u),
\end{equation}
it follows that the inner boundary integral becomes
\begin{align}\label{innerboundary}
\begin{split}
\int_{\partial_{\neq 0}M_{ext}}\left(\partial_{\upsilon}|\nabla u|+k(\nabla u,\upsilon)\right)dA=&\int_{\partial M_{ext}}\left[H|\upsilon(u)|
-\left(\mathrm{Tr}_{\partial M_{ext}}k\right)\upsilon(u)\right]dA\\
=&\sum_{i=1}^n \int_{\partial_i M_{ext}}\theta_{\pm}|\upsilon(u)|dA,
\end{split}
\end{align}
where we have used \eqref{boundarybehavior} in the last step. The notation $\theta_{\pm}$ above indicates that the integrand contains $\theta_+$ for a MOTS
component and $\theta_-$ for a MITS component. We conclude that the inner boundary integral vanishes. Similarly, if the boundary of the generalized exterior region consists of weakly trapped surfaces then this boundary integral is nonpositive.

\subsection{Proof of Theorem \ref{main} and Corollary \ref{main1} (the inequality)}

By combining \eqref{INEQ1}, \eqref{outerboundary}, \eqref{innerboundary}, and taking the limit as $L\rightarrow \infty$ we obtain
\begin{equation}
E+P_1 \geq\frac{1}{16\pi}\int_{M_{ext}}
\left(\frac{|\bar{\nabla}^2 u|^2}{|\nabla u|}+2(\mu-|J|_g)|\nabla u|\right)dV,
\end{equation}
since $q>\frac{1}{2}$. Furthermore, it may be assumed without loss of generality that the ADM linear momentum satisfies $P_1=-|P|$, by applying an appropriate rotation of the asymptotically flat coordinates $\tilde{x}$. This yields the desired inequalities \eqref{mainlowerb} and \eqref{mainlowerb5}.

\section{The Case of Equality}
\label{sec6} \setcounter{equation}{0}
\setcounter{section}{7}

In this section we will finish the proof of Theorem \ref{main} and Corollary \ref{main1} by establishing the rigidity statement. Namely, it will be shown that if the dominant energy condition holds and $E=|P|$, then $E=|P|=0$, $M$ is diffeomorphic to $\mathbb{R}^3$, and the initial data set $(M,g,k)$ arises from an isometric embedding into Minkowski space. The proof of inequality \eqref{mainlowerb} shows that if $E=|P|$, then $|\bar{\nabla}^2 u|=0$ on $M_{ext}$ for some asymptotically linear spacetime harmonic function. This in turn guarantees that the gradient never vanishes.

\begin{lemma}\label{l:nocrit}
Suppose that $u\in C^2(M_{ext})$ satisfies
\begin{equation}
\begin{cases}
\bar{\nabla}_{ij}u=\nabla_{ij} u+k_{ij}|\nabla u|=0&\text{ in }M_{ext},\\
u=v +O(r^{1-2q})&\text{ in }M_{end},
\end{cases}
\end{equation}
where $v\neq 0$ is a solution of \eqref{vequation1} and \eqref{vequationb1} in $M_{ext}$. Then there exists a constant $c>0$ such that $|\nabla u|\geq c$ on $M_{ext}$.
\end{lemma}

\begin{proof}
Due to the asymptotics in the asymptotically flat end, there exists a large coordinate sphere $S_{r_0}$ such that $|\nabla u|\geq \tfrac{1}{2}$ holds on the unbounded component of $M_{ext}\setminus S_{r_0}$. Denote the bounded component of $M_{ext}\setminus S_{r_0}$ by $M_{r_0}$, and pick $x_0 \in S_{r_0}$. Given $x\in M_{r_0}$, let $\gamma\subset M_{r_0}$ be a curve parameterized by arclength connecting $x_0$ to $x$.
Observe that since
\begin{equation}
|\nabla|\nabla u||\leq |\nabla^2 u|\leq |k||\nabla u|,
\end{equation}
we have that
\begin{equation}
\left|\left(\log|\nabla u|\circ\gamma\right)'\right|\leq |\nabla\log|\nabla u||\circ\gamma
\leq C(1+r\circ\gamma)^{-1-q},
\end{equation}
where \eqref{asymflat} was used and it is assumed that the radial function $r$ in the asymptotically flat end is extended smoothly to a positive function on all of $M_{ext}$. By integrating along $\gamma$, it follows that there is a constant $C_1>0$ such that
\begin{equation}
C_1^{-1}|\nabla u(x_0)|\leq|\nabla u(x)|\leq C_1|\nabla u(x_0)|,\quad\quad\quad x\in M_{r_0}.
\end{equation}
Clearly $C_1$ may be chosen independently of $x$. The desired result follows since $|\nabla u(x_0)|\geq\tfrac{1}{2}$.
\end{proof}

Nonvanishing of the gradient is inconsistent with the boundary behavior of  spacetime harmonic functions given by Theorem \ref{Thm:PDE} and Lemma \ref{l:dirichlet} (1). This implies that the generalized exterior region has no boundary, and leads to trivial topology for the initial data.

\begin{prop}\label{r3}
Let $(M,g,k)$ be an asymptotically flat initial data set for the Einstein equations satisfying the dominant energy condition. If $E=|P|$ in one of the asymptotic ends, then $M$ is diffeomorphic to $\mathbb{R}^3$.
\end{prop}

\begin{proof}
As described above, the generalized exterior region associated with a designated end $M_{end}$ satisfying $E=|P|$ must have empty boundary. This fact can be applied repeatedly in order to conclude the desired result. The arguments follow closely those of \cite[Theorem 5.1]{EichmairGallowayPollack}, and thus here we simply outline the main steps. We first claim that $M$ has a single asymptotically flat end. Suppose instead that it has other nondesignated ends. Then large coordinate spheres in these nondesignated ends are outer trapped from the point of view of the designated end, that is, they satisfy $\theta_+ <0$; this null expansion is computed with respect to the unit normal pointing away from the nondesignated end. Furthermore, a large coordinate sphere in the designated end is untrapped, that is, it satisfies $\theta_+>0$; this null expansion is computed with respect to the unit normal pointing towards the designated end. These surfaces may be used as barriers to produce a MOTS in the region that they bound, see \cite{AnderssonMetzger,Eichmair-1}. The existence of this MOTS implies that the generalized exterior region for the designated end must have a nonempty boundary, contradicting the fact that $E=|P|$ in this end. It follows that $M$ has a single end.

Suppose now that $M$ is not orientable. Then it possesses a connected orientable double cover, endowed with the pullback data coming from $(g,k)$. This cover has two ends, and therefore as above there exists a MOTS. Since $E=|P|$ in each end, this is a contradiction. We conclude that $M$ is orientable. It follows that $M\cong N\# \mathbb{R}^3$, where $N$ is a compact orientable 3-manifold without boundary. According to \cite{Hempel} and the resolution of the geometrization conjecture,
there is a normal subgroup of $\pi_1(N)=\pi_1(M)$ with finite index that does not contain a given non-identity element; such groups are referred to as residually finite. Thus, if $M$ is not simply connected then it possesses a nontrivial finite sheeted cover. As before, the fact that this covering has multiple ends leads to a contradiction with $E=|P|$. We conclude that $M$ is simply connected. The positive resolution of the Poincar\'{e} conjecture then shows that $M\cong\mathbb{R}^3$.
\end{proof}

We will now establish the main result of this section, and complete the proof of Theorem \ref{main} and Corollary \ref{main1}. The proof is motivated by Beig and Chru\'{s}ciel's treatment \cite[Theorem 4.1]{BeigChrusciel} of the spinorial approach to the rigidity statement. One advantage of the approach presented here is that the isometric embedding into Minkowski space may be presented explicitly as a graph given by a linear combination of spacetime harmonic functions.

\begin{theorem}
Let $(M,g,k)$ be an asymptotically flat initial data set for the Einstein equations satisfying the dominant energy condition. If $E=|P|$ in one of the asymptotic ends, then $E=|P|=0$ and $(M,g,k)$ arises from an isometric embedding into Minkowski space.
\end{theorem}

\begin{proof}
Proposition \ref{r3} shows that $M\cong\mathbb{R}^3$, so there is only one end. Let $x=(x^1,x^2,x^3)$ be a global coordinate system which coincides with the asymptotically flat coordinates in this end, and denote the spacetime harmonic function of Theorem \ref{Thm:PDE} which is asymptotic to $a_i x^i$ by $u(a_1,a_2,a_3)$; it is assumed as usual that $\sum_i a_i^2 =1$.
Define a lapse function $\alpha$ and shift vector $\beta$ by
\begin{equation}
\alpha=\left|\nabla u\left(\frac{1}{\sqrt{2}},\frac{1}{\sqrt{2}},0\right)\right|+\left|\nabla u\left(-\frac{1}{\sqrt{2}},0,\frac{1}{\sqrt{2}}\right)\right|-\left|\nabla u\left(0,\frac{1}{\sqrt{2}},\frac{1}{\sqrt{2}}\right)\right|,
\end{equation}
and
\begin{equation}
\beta=\nabla u\left(\frac{1}{\sqrt{2}},\frac{1}{\sqrt{2}},0\right)+\nabla u\left(-\frac{1}{\sqrt{2}},0,\frac{1}{\sqrt{2}}\right)-\nabla u\left(0,\frac{1}{\sqrt{2}},\frac{1}{\sqrt{2}}\right).
\end{equation}
Observe that $\alpha\rightarrow 1$ and $|\beta|\rightarrow 0$ as $r\rightarrow\infty$. From these we may form the stationary spacetime
$(\mathbb{R}\times M, \bar{g})$ where
\begin{equation}
\bar{g}=-\left(\alpha^2 -|\beta|^2\right)dt^2 +2\beta_i dx^i dt +g.
\end{equation}
This is the Killing development of $(M,g,k,\alpha,\beta)$ with Killing initial data $(\alpha,\beta)$ decomposing the Killing vector
\begin{equation}
\partial_t =\alpha \mathbf{n}+\beta,
\end{equation}
where $\mathbf{n}$ is the unit normal to the hypersurfaces $t=const$. It will be shown that this spacetime is isometric to Minkowski space, and that $(M,g,k)$ arise from a constant time slice.

Next observe that the inequality $E\geq |P|$ implies that $E=|P|=0$ from Huang and Lee's result \cite{HuangLee}.
The condition $E=|P|=0$ together with \eqref{mainlowerb} then produces
\begin{equation}
\nabla_{ij}u(a_1,a_2,a_3)=-\left|\nabla u(a_1,a_2,a_3)\right|k_{ij},
\end{equation}
and therefore
\begin{equation}
\nabla_i \beta_j =-\alpha k_{ij},\quad\quad\quad \partial_i \alpha=-\beta^j k_{ij}.
\end{equation}
It follows that
\begin{equation}
\frac{1}{2}\partial_i(\alpha^2-|\beta|^2)=\alpha\partial_i \alpha-\beta^j\nabla_i \beta_j=0.
\end{equation}
Since $\alpha^2 -|\beta|^2 \rightarrow 1$ as $r\rightarrow\infty$, we then have that $\alpha^2-|\beta|^2 \equiv 1$
and
\begin{equation}
\bar{g}=-dt^2 +2\beta_i dx^i dt +g=-\left(dt-\beta_i dx^i\right)^2 +\left(g_{ij} +\beta_i \beta_j\right)dx^i dx^j.
\end{equation}
This simplification shows that Christoffel symbols involving the time coordinate vanish. Namely, if $a,b,c=0,1,2,3$ with the index $0$ representing the time coordinate, then
\begin{equation}\label{christoffel}
\bar{\Gamma}_{at}^b =\frac{1}{2}\bar{g}^{bc}\left(\partial_a \bar{g}_{ct}+\partial_t \bar{g}_{ca}-\partial_c \bar{g}_{at}\right)=\frac{1}{2}\bar{g}^{bc}\left(\partial_a \beta_c -\partial_c \beta_a\right)=0,
\end{equation}
where we have used that $\beta=d\mathbf{u}$ with
\begin{equation}
\mathbf{u}=u\left(\frac{1}{\sqrt{2}},\frac{1}{\sqrt{2}},0\right)+ u\left(-\frac{1}{\sqrt{2}},0,\frac{1}{\sqrt{2}}\right)- u\left(0,\frac{1}{\sqrt{2}},\frac{1}{\sqrt{2}}\right).
\end{equation}
The second fundamental form of the constant time slice $t=const$ is then
\begin{equation}
\langle\bar{\nabla}_i\mathbf{n},\partial_j \rangle
=\alpha^{-1}\langle\bar{\nabla}_i(\partial_t -\beta),\partial_j \rangle
= \alpha^{-1}\bar{\Gamma}^b_{it} \bar{g}_{bj} -\alpha^{-1}\nabla_i \beta_j=k_{ij}.
\end{equation}
Furthermore, the initial data metric $g$ is the induced metric on $t=const$ and hence $(M,g,k)$ arises from a
constant time slice in this Killing development.

Let us now show that the Killing development is flat. Consider the null vector fields
\begin{equation}
X_\ell =\nabla u_\ell + |\nabla u_\ell|\mathbf{n},\quad\quad\ell=1,2,3,
\end{equation}
where
\begin{equation}
u_1=u(1,0,0),\quad\quad\quad u_2=u(0,1,0),\quad\quad\quad u_3 =u(0,0,1),
\end{equation}
and these functions are extended trivially in the $t$-direction to all of $\mathbb{R}\times M$. Let $i,j=1,2,3$ and compute
\begin{equation}
\langle \bar{\nabla}_t X_{\ell},\partial_{t}\rangle
=\langle \bar{\nabla}_t \nabla u_{\ell},\partial_t\rangle
+|\nabla u_{\ell}|\langle\bar{\nabla}_t \partial_t ,\partial_t \rangle
=0,
\end{equation}
\begin{equation}
\langle\bar{\nabla}_i X_{\ell},\partial_j\rangle=\langle \bar{\nabla}_i \nabla u_{\ell},\partial_j\rangle+|\nabla u_{\ell}|\langle \bar{\nabla}_i\mathbf{n},\partial_j\rangle=\nabla_{ij} u_{\ell}+|\nabla u_{\ell}| k_{ij}=0,
\end{equation}
\begin{equation}
\langle\bar{\nabla}_t X_{\ell},\partial_i\rangle=\langle\bar{\nabla}_{t}\nabla u_{\ell},\partial_i\rangle+|\nabla u_{\ell}|\langle \bar{\nabla}_t\mathbf{n},\partial_i\rangle=-|\nabla u_{\ell}|\langle\mathbf{n},\bar{\nabla}_t \partial_i\rangle
=0,
\end{equation}
and
\begin{align}
\begin{split}
\langle\bar{\nabla}_i X_{\ell},\partial_t\rangle=&\partial_i \langle X_{\ell},\partial_t \rangle\\
=&\partial_i\left(u^j_{\ell} \bar{g}_{jt}-\frac{|\nabla u_{\ell}|}{\alpha}\left(1+\beta^j \bar{g}_{jt}\right)\right)\\
=&\partial_i \left( u^j_{\ell} \beta_j -\alpha |\nabla u_{\ell}|\right)\\
=&\beta^j \nabla_{ij} u_{\ell}+u^j_\ell \nabla_{i}\beta_{j}
-|\nabla u_{\ell}|\partial_i \alpha-\alpha \partial_i |\nabla u_{\ell}|\\
=&0.
\end{split}
\end{align}
It follows that these vector fields are covariantly constant, and in light of \eqref{christoffel} the same holds for the Killing field $\partial_t$.
Furthermore, the collection of vector fields $\{X_1, X_2, X_3, \partial_t\}$ is linearly independent in the asymptotic end, and by parallel translation this property holds globally. We then have that the spacetime $(\mathbb{R}\times M,\bar{g})$ is flat.

In order to complete the proof it must be shown that the Killing development is isometric to Minkowski space. To this end, observe that since $\beta$ is exact, the change of coordinates $\bar{t}=t-\mathbf{u}(x)$, $\bar{x}=x$ yields the static structure
\begin{equation}
\bar{g}=-d\bar{t}^2 + (g+d\mathbf{u}^2).
\end{equation}
The manifold $(\mathbb{R}^3, g+d\mathbf{u}^2)$ is asymptotically flat, and hence complete. Moreover it is flat, and thus isometric to Euclidean 3-space, yielding the desired conclusion. Next observe that since the initial data $(M,g,k)$ arise from the $t=0$ slice, we find that they may be expressed as the graph of a linear combination of spacetime harmonic functions, namely $\bar{t}=-\mathbf{u}(\bar{x})$.
%Note that $|\nabla \mathbf{u}|^2=O(r^{-2q})$ as $r\rightarrow\infty$, which is consistent with vanishing ADM energies %$E(g+d\mathbf{u}^2)=E(g)$, so that $|P|=E(g)=0$.
\end{proof}

\bigskip
\noindent\textbf{Acknowledgements.} The authors would like to thank Hubert Bray, Greg Galloway, and Daniel Stern for helpful comments.

\end{document}